\UseRawInputEncoding
\documentclass[12pt, reqno]{amsart}
\usepackage[margin=1in]{geometry}
\usepackage{amssymb,latexsym,amsmath,amscd,amsfonts}
\usepackage{latexsym}
\usepackage[mathscr]{eucal}
\usepackage{bm}
\usepackage{mathptmx}
\usepackage{tabularx}
\usepackage{amssymb}
\usepackage{tikz}
\usepackage{amsthm}
\usepackage{ragged2e}
\usepackage{dcolumn}
\usepackage[all]{xy}
\usepackage{xcolor}
\usepackage{enumitem}
\usepackage[utf8]{inputenc}

\def \qed {\hfill \vrule height6pt width 6pt depth 0pt}
\def\textmatrix#1&#2\\#3&#4\\{\bigl({#1 \atop #3}\ {#2 \atop #4}\bigr)}
\def\dispmatrix#1&#2\\#3&#4\\{\left({#1 \atop #3}\ {#2 \atop #4}\right)}
\newcommand{\beg}{\begin{equation}}
	\newcommand{\eeg}{\end{equation}}
\newcommand{\ben}{\begin{eqnarray*}}
	\newcommand{\een}{\end{eqnarray*}}

\newtheorem{thm}{Theorem}[section]
\newtheorem{cor}[thm]{Corollary}
\newtheorem{lem}[thm]{Lemma}

\newtheorem{prop}[thm]{Proposition}
\numberwithin{equation}{section} \theoremstyle{definition}
\newtheorem{defn}[thm]{Definition}
\newtheorem{rem}[thm]{Remark}

\newtheorem{eg}[thm]{Example}

\newcommand{\E}{\mathbb{E}}
\newcommand{\EC}{\overline{\mathbb{E}}}
\newcommand{\Q}{\mathbb{H}}
\newcommand{\C}{\mathbb{C}}
\newcommand{\B}{\mathbb{B}}
\newcommand{\G}{\mathbb{G}}
\newcommand{\Gg}{\mathbb{G}_2}
\newcommand{\D}{\mathbb{D}}
\newcommand{\T}{\mathbb{T}}
\newcommand{\N}{\mathbb{N}}

\newcommand{\Pe}{\mathbb{P}}
\newcommand{\HS}{\mathcal{H}}
\newcommand{\CQ}{\overline{\mathbb{H}}}
\newcommand{\PC}{\overline{\mathbb{P}}}
\newcommand{\BC}{\overline{\mathbb{B}}_2}
\newcommand{\DC}{\overline{\mathbb{D}}}
\newcommand{\UT}{\underline{T}}

\title[Operators associated with the hexablock]{Operators associated with a domain in $\mathbb C^4$ and applications}
\author[Pal and Tomar]{Sourav Pal and Nitin Tomar}

\address[Sourav Pal]{Mathematics Department, Indian Institute of Technology Bombay,
	Powai, Mumbai - 400076, India.} \email{souravmaths@gmail.com, sourav@math.iitb.ac.in}

\address[Nitin Tomar]{Mathematics Department, Indian Institute of Technology Bombay, Powai, Mumbai-400076, India.} \email{tomarnitin414@gmail.com}		



\begin{document}
	
	\maketitle

	\begin{abstract}
		
		We introduce operator theory on the \textit{hexablock}, a domain in $\mathbb C^4$ defined by
		\[
		\mathbb H= \left\{(a, x_1, x_2, x_3) \in \mathbb{C} \times \mathbb{E} : \left|\frac{a\sqrt{(1-|z_1|^2)(1-|z_2|^2)}}{1-x_1z_1-x_2z_2+x_3z_1z_2}\right|<1 \ \text{whenever} \ |z_1|, |z_2|<1 \right\},
		\]
where $\mathbb E \subset \mathbb C^3$ is a domain named \textit{tetrablock} which is given by
		\[
		\mathbb E=\left\{(x_1, x_2, x_3) \in \mathbb C^3 \ : \ 1-x_1z_1-x_2z_2+x_3z_1z_2 \ne 0 \ \text{whenever} \ |z_1| \leq 1, |z_2| \leq 1 \right\}.
		\]	
A commuting quadruple of Hilbert space operators $(A, X_1, X_2, X_3)$ is said to be an \textit{$\mathbb H$-contraction} if the closed hexablock $\overline{\mathbb H}$ is a spectral set for $(A, X_1, X_2, X_3)$. A commuting quadruple $(A, X_1, X_2, X_3)$ of normal operators is said to be an \textit{$\mathbb H$-unitary} if the joint spectrum $\sigma_T(A, X_1, X_2, X_3)$ of $(A, X_1, X_2, X_3)$ is contained in the distinguished boundary $b\mathbb H$ of $\overline{\mathbb H}$. Also, $(A, X_1, X_2, X_3)$ is called an \textit{$\mathbb H$-isometry} if it is the restriction of an $\mathbb H$-unitary $(\widehat{A}, \widehat{X}_1, \widehat{X}_2, \widehat{X}_3)$ to a joint invariant subspace for $\widehat{A}, \widehat{X}_1, \widehat{X}_2, \widehat{X}_3$. In this article, we develop a theory of $\Q$-contractions and realize the operator theory of the following five domains within the unifying framework of $\Q$-contractions: Euclidean biball $\mathbb B_2$, bidisc $\mathbb D^2$,  symmetrized bidisc $\mathbb G_2$, tetrablock $\mathbb E$ and pentablock $\mathbb P$. Several characterizations for the $\mathbb H$-unitaries and $\mathbb H$-isometries are provided. We prove that every $\mathbb H$-isometry admits a Wold type decomposition into an $\mathbb H$-unitary and a pure $\mathbb H$-isometry. A conditional dilation is obtained for a $\mathbb H$-contraction $(\widehat{A}, \widehat{X}_1, \widehat{X}_2, \widehat{X}_3)$ with $X_3$ being a $C_{.0}$ contraction, i.e., ${X_3^*}^n \rightarrow 0$ strongly as $n \rightarrow \infty$. Using the operator model of a $C_{.0}$ contraction, we construct a functional model for such $\mathbb H$-contractions in terms of commuting Toeplitz operators. Moreover, we find an equivalent formulation of the rational dilation problem for biball, tetrablock and pentablock via rational dilation on the hexablock. Finally, we show that every $\mathbb H$-contraction orthogonally decomposes into an $\mathbb H$-unitary and a completely non-unitary $\mathbb H$-contraction, which is analogous to canonical decomposition of a contraction. 
	\end{abstract}

	\section{Introduction} \label{sec_intro}
	
	
	\noindent Throughout the paper, all operators are bounded linear maps acting on complex Hilbert spaces. For a Hilbert space $\HS$, we denote by $\mathcal{B}(\HS)$ the algebra of operators on $\HS$. The symbol $\C$ stands for the set of complex numbers. We denote by $\D$ and $\T$ the unit disk and the unit circle, respectively with center at the origin in $\C$. A contraction is an operator with norm at most $1$. For a contraction $T$, we denote by $D_T$ and $\mathcal{D}_T$, the defect operator $(I-T^*T)^{1\slash 2}$ and the defect space $\overline{Ran} \, D_T$ of $T$, respectively. For an operator $T$, $|T|$ stands for $(T^*T)^{\frac{1}{2}}$ and $\omega(T)$ denotes the numerical radius of $T$. The commutator of operators $A,B$ is given by $[A, B]=AB-BA$. For a commuting tuple of operators $\underline{T}=(T_1, \dotsc, T_n)$ acting on a Hilbert space $\HS$, we denote by $\sigma_T(T_1, \dotsc, T_n)$ the polynomial joint spectrum (or, the joint spectrum) of $(T_1, \dotsc, T_n)$ relative to the closed algebra $\mathcal{A}$ of $\mathcal{B}(\mathcal{H})$ generated by $T_1, \dotsc, T_n$ and the identity operator $I$ on $\HS$, i.e., 
	\[
	\sigma_T(T_1, \dotsc, T_n)=\{(\lambda_1, \dotsc, \lambda_n) \in \C^n : I \notin (T_1-\lambda_1)\mathcal{A}+\dotsc+(T_n-\lambda_n)\mathcal{A} \}.
	\]
In this article, we introduce operator theory on the hexablock $\mathbb{H}$, a domain in $\mathbb{C}^4$ that naturally arises from a special case of $\mu$-synthesis problem in control theory. The \emph{hexablock} is defined by
\begin{equation}\label{eqn_101}
		\mathbb H= \left\{(a, x_1, x_2, x_3) \in \C \times \E :  \sup_{z_1, z_2 \in \D}\left|\frac{a\sqrt{(1-|z_1|^2)(1-|z_2|^2)}}{1-x_1z_1-x_2z_2+x_3z_1z_2}\right|<1 \right\},
	\end{equation}
	where $\E$ is the tetrablock, a domain in $\mathbb C^3$ given by 
	\begin{equation}\label{eqn:Intro-new-01}
	\mathbb E=\{(x_1, x_2, x_3) \in \mathbb C^3 \ : \ 1-x_1z_1-x_2z_2+x_3z_1z_2 \ne 0 \ \
	\text{whenever} \ \
	|z_1| \leq 1, |z_2| \leq 1 \}.
	\end{equation}	
The aim of this article is twofold. We develop operator theory on $\Q$ that is analogous to the famous Nagy-Foias theory for Hilbert space contractions. Also, we realize operator theory of five different domains within a unifying framework of operators associated with the hexablock. These five domains are the following: the Euclidean biball $\B_2$, the bidisc $\mathbb{D}^2$, the symmetrized bidisc $\mathbb G_2$, the tetrablock $\E$ and the pentablock $\Pe$, where
\begin{equation}\label{eqn:Intro-new-02}
	\mathbb G_2  = \{ (z_1+z_2,z_1z_2) : \; |z_1|, |z_2|<1 \} \ \  \text{and} \ \ \mathbb P  = \left\{(a, s, p) \in \C \times \G_2 \ : \ \sup_{z \in \D}\left|\frac{a(1-|z|^2)}{1-sz+pz^2}\right|<1 \right\}.
\end{equation}
The symmetrized bidisc $\mathbb G_2$, the tetrablock $\E$ and the pentablock $\E$ arise naturally in connection with special cases of $\mu$-synthesis. The $\mu$-\textit{synthesis} problem is a part of robust control theory (e.g., see \cite{Doyle, Francis}) of systems consisting of interconnected devices whose outputs are linearly dependent on the inputs. Given a linear subspace $E$ of $M_n(\mathbb C)$, the space of all $n \times n$
complex matrices, the functional
\begin{equation} \label{eqn:NEW-01}
\mu_E(A):= (\text{inf} \{ \|X \|: X\in E \text{ and } (I-AX)
\text{ is singular } \})^{-1} \quad \text{($A\in  M_n(\mathbb C)$)}
\end{equation}
is called the \textit{structured singular value} and the linear subspace $E$ is referred to as the \textit{structure}. Given distinct points $\alpha_1, \dots , \alpha_d$ in $\D$ and matrices $B_1, \dots, B_d \in M_n(\C)$, the goal of $\mu$-synthesis is to find a holomorphic map $F:\mathbb D \rightarrow M_n(\mathbb C)$ subject to finitely many interpolation conditions such that $\mu_E(F(\lambda))<1$ for all $\lambda \in \mathbb D$. If $E= M_n(\mathbb C)$, then $\mu_E (A)=\|A\|$, while if $E$ is the space of all scalar multiples of the identity matrix, then $\mu_E(A)$ is the spectral radius $r(A)$. For any linear subspace $E$ of
$M_n(\mathbb C)$ that contains the identity matrix, $r(A)\leq \mu_E(A) \leq \|A\|$. The $\mu$-synthesis with respect to the linear subspace $E \subset M_2(\C)$ consisting of scalar multiples of the identity matrix gives rise to the symmetrized bidisc $\mathbb{G}_2$ (see \cite{AglerII}), whereas the tetrablock $\mathbb{E}$ was induced by the subspace of $2 \times 2$ diagonal matrices, e.g., see \cite{Abouhajar}. Also, the pentablock $\mathbb{P}$ is originated in the $\mu$-synthesis of the $\mu_E$-unit ball, where $E=\left\{ \begin{pmatrix}
a & b \\
0 & a
\end{pmatrix} \, : \, a, b \in \C
\right\}$,
e.g., see \cite{Agler}. We now show the readers by an example how the $\mu$-synthesis problem with respect to a linear subspace $E$ of $M_n(\C)$ gives rise to a domain. Suppose $E=\left\{\begin{pmatrix}
z_1 & 0 \\
0 & z_2
\end{pmatrix} : z_1, z_2 \in \mathbb{C} \right\}$. It was shown in \cite{Abouhajar} how the $\mu$-synthesis with respect to this subspace induces the domain \textit{tetrablock}. Indeed, for $X=\begin{pmatrix}
z_1 & 0\\
0& z_2
\end{pmatrix} \in E$ and $A=(a_{ij})_{i,j=1}^2 \in M_2(\mathbb{C})$, we have that
$\|X\|=\max\{|z_1|, |z_2|\}$ and $\det(I-AX)=1-a_{11}z_1-a_{22}z_2+\det(A)z_1z_2$ and it follows from the definition of $\mu_E$ (see Equation-(\ref{eqn:NEW-01})) that
\begin{align*}
\mu_E(A)<1
& \iff \|X\|>1 \ \text{for all} \ X \in E \ \text{with} \ \det(I-AX)=0\\
& \iff \det(I-AX)\ne 0 \ \text{for all} \ X \in E \ \text{with} \ \|X\| \leq 1\\
& \iff 1-a_{11}z_1-a_{22}z_2+\det(A)z_1z_2\ne 0 \ \text{for all} \ z_1, z_2 \in \overline{\mathbb{D}}\\
& \iff (a_{11}, a_{22}, \det(A)) \in \mathbb{E}.
\end{align*}
Thus, $A=(a_{ij})_{i,j=1}^2$ is in the $\mu_E$-unit ball if and only if the point $(a_{11}, a_{22}, \det(A))$ belongs to the tetrablock $\mathbb E$. The authors of this article and Biswas discover in \cite{Biswas} the domain hexablock from the $\mu$-synthesis subject to the linear subspace of all $2 \times 2$ upper triangular matrices. In the same article, they identifies all three previous domains (that are $\mathbb{G}_2$, $\mathbb{E}$ and $\mathbb{P}$) as special cases of the hexablock $\Q$. These three domains have received considerable attention in recent years from complex-geometric, function-theoretic, and operator-theoretic perspectives; 
see, for example, \cite{AglerII, AglerIV, AglerVI, AglerVII, Bhattacharyya} for $\mathbb{G}_2$, \cite{Abouhajar, Alsalhi, Bhattacharyya-01, Pal2016, Pal-IV} for $\mathbb{E}$, and \cite{Agler, JindalII, Kosinski, PalN, Su, SuII} for $\mathbb{P}$, and references therein. For the hexablock, the structure $E$ is chosen to be the linear subspace of all $2 \times 2$ upper triangular matrices. With this choice of $E$, the authors of \cite{Biswas} introduced the sets
	\[
	\mathbb{H}_{\mu}=\{\pi(A) : \mu_E(A)<1\}\quad \text{and} \quad \mathbb{H}_{N}=\{\pi(A) : \|A\|<1\},
	\]
	where $\pi(A)=(a_{21}, a_{11}, a_{22}, \det(A))$ for $A=[a_{ij}] \in M_2(\C)$. However, it turns out that  $\mathbb{H}_{\mu} \ne \mathbb{H}_N$ and neither of these sets is a domain in $\C^4$ unlike the previous $\mu$-synthesis cases that produced the domains $\Gg, \E$ and $\Pe$. This deviation in the $\mu$-synthesis problem unlike the previously studied ones makes the theory of hexablock a little complicated but interesting from the function-theoretic point of view as discussed in detail in Chapters 1 and 2 of \cite{Biswas}. In this context, a fundamental problem is to extract a domain in $\C^4$ that contains $\mathbb{H}_{N}$ as well as $\mathbb{H}_{\mu}$. The domain that arises in this connection is hexablock, which is defined in \eqref{eqn_101}. In fact, it is proved in \cite{Biswas} that $\mathbb{H}$ is same as the interior of the closure of $\mathbb{H}_\mu$, i.e., $\mathbb{H}=\text{int}(\overline{\mathbb{H}}_{\mu})$. Moreover, $\mathbb{H}=\text{int}( \widehat{\overline{\mathbb{H}}_N})$, the interior of polynomial convex hull of the closure of $\mathbb{H}_N$. This establishes a strong connection among the sets $\mathbb{H}, \mathbb{H}_\mu$ and  $\mathbb{H}_N$. Biswas et al. studied further the complex-theoretic and function-theoretic properties  of the hexablock in \cite{Biswas}. In this article, our primary object of study is a commuting quadruple of Hilbert space operators having the closed hexablock as a spectral set.	
	
\begin{defn}\label{defn_contr}
	Let $X \subseteq \C^n$ be a polynomially convex compact set. Then $X$ is called a \textit{spectral set} for a commuting tuple of operators $\underline{T}=(T_1, \dotsc, T_n)$ if von Neumann's inequality holds on $X$ for all holomorphic polynomials in $n$-variables, i.e., for every $p \in \C[z_1, \dotsc, z_n]$,
	\begin{equation*}
		\|p(T_1, \dotsc, T_n) \| \leq \sup \{|p(z_1, \dotsc, z_n)| : (z_1, \dotsc, z_n)\in X\} = \|p\|_{\infty, \, X \,}.
	\end{equation*}	
	Moreover, $X$ is said to be a \textit{complete spectral set} for $(T_1, \dotsc, T_n)$ if  
	\begin{equation}\label{eqn_vN_matrix}
	\|f(T_1, \dotsc, T_n) \| \leq \sup \{\|f(z_1, \dotsc, z_n)\| : (z_1, \dotsc, z_n)\in X\}
	\end{equation}
	holds for every matricial polynomial $f=[f_{ij}]$, where each $f_{ij} \in \C[z_1, \dotsc, z_n]$. Let $\Omega \subseteq \mathbb{C}^n$ be a bounded domain such that $\overline{\Omega}$ is polynomially convex. A commuting $n$-tuple of operators $\underline{T}$ is said to be an \textit{$\Omega$-contraction} (or, \textit{$\overline{\Omega}$-contraction}) if $\overline{\Omega}$ is a spectral set for $\underline{T}$.
\end{defn}
	
	Unitaries, isometries and co-isometries are special classes of contractions. A unitary is a normal operator having its spectrum on the unit circle $\T$. An isometry is the restriction of a unitary to an invariant subspace and a co-isometry is the adjoint of an isometry. In an analogous manner, we define unitary, isometry and co-isometry associated with the hexablock. To do so, we briefly describe the definition of distinguished boundary. For a bounded domain $\Omega  \subset \C^n$, the distinguished boundary of $\overline{\Omega}$ is the smallest closed subset $b\Omega$ (or, $b\overline{\Omega})$ of $\overline{\Omega}$ such that every function that is continuous on $\overline{\Omega}$ and analytic in $\Omega$ attains its maximum modulus on $b\Omega$.
	
	\begin{defn}\label{defn:H-unitary}
		A commuting quadruple of operators $(A, X_1, X_2, X_3)$ on a Hilbert space $\HS$ is called 
		\begin{itemize}
			
			\item[(i)] an $\mathbb{H}$-\textit{unitary} if $A, X_1, X_2, X_3$ are normal operators and the polynomial joint spectrum is contained in the distinguished boundary of the hexablock;
						
			\item[(ii)] an $\mathbb{H}$-\textit{isometry} if there is a Hilbert space $\mathcal K \supseteq \HS$ and an $\mathbb{H}$-unitary $(\widehat A, \widehat X_1, \widehat X_1, \widehat X_3)$ on $\mathcal K$ such that $\HS$ is a joint invariant subspace for $\widehat A, \widehat X_1, \widehat X_2, \widehat X_3$ and 
			\[
			(A, X_1, X_2, X_3)=(\widehat A|_{\HS}, \widehat X_1|_{\HS}, \widehat X_2|_{\HS}, \widehat X_3|_{\HS});
			\]
			
			\item[(iii)] an $\mathbb{H}$-\textit{co-isometry} if $(A^*, X_1^*, X_2^*, X_3^*)$ is an $\mathbb{H}$-isometry. 
		\end{itemize}
	Similarly, one can define unitary, isometry and co-isometry for the classes of $\Pe$-contractions, $\B_2$-contractions, $\E$-contractions and $\Gamma$-contractions. Moreover, an isometry (on $\HS$) associated with a domain is called \textit{pure} if there is no nonzero proper joint reducing subspace of the isometry on which it acts like a unitary associated with the domain.
	\end{defn}
Note that one can define unitary with respect to a bounded domain whose closure is polynomially convex (e.g., $\Q$-unitary as in Definition \ref{defn:H-unitary} for the domain hexablock) by considering the Taylor joint spectrum instead of the polynomial joint spectrum. These two notions of joint spectrum coincide in the context of this paper as we deal here with domains having polynomially convex closures, such as pentablock $\Pe$, Euclidean unit ball $\B_2$ in $\C^2$, tetrablock $\E$ and symmetrized bidisc $\G_2$. Thus, the delicate issues surrounding various notions of joint spectrum are not relevant to this paper. The starting point of the present work is the following result, which highlights the connections of the hexablock with the biball, the tetrablock, the symmetrized bidisc, the bidisc and the pentablock.

\begin{thm}[\cite{Biswas}, Chapter 10]\label{thm_connect}
	Let $(a, x_1, x_2, x_3, s, p) \in \C^6$. Then the following holds:
	\begin{enumerate}
		\item $(a, x_1) \in \BC$ if and only if $(a, x_1, 0, 0) \in \CQ$.
	  \item $(x_1, x_2, x_3) \in \EC$ if and only if $(0, x_1, x_2, x_3) \in \CQ$.
		\item $(a, s, p) \in \PC$ if and only if $(a, s\slash 2, s\slash 2, p) \in \CQ$.
		\item $(s, p) \in \Gamma$ if and only if $(0, s\slash 2, s\slash 2, p) \in \CQ$. 
		\item $(a, x_3) \in \DC^2$ if and only if $(a, 0, 0, x_3) \in \CQ$.
	\end{enumerate}
\end{thm}
In Section \ref{sec_03} of this paper, we prove an operator theoretic analogue of Theorem \ref{thm_connect}. In particular, we show how $\Pe$-contractions, $\B_2$-contractions, $\E$-contractions, $\Gamma$-contractions and even a commuting pair of contractions can be realized within the framework of $\Q$-contractions, thereby placing these classes in a unified setting. Furthermore, it is proved that if $(A, X_1, X_2, X_3)$ is an $\Q$-contraction, then $(A, X_1), (A,X_2)$ are $\B_2$-contractions, $(X_1, X_2, X_3)$ is an $\E$-contraction and $(A, X_1+X_2, X_3)$ is a $\Pe$-contraction. However, a converse to this result does not hold which we establish via a counter example in the same section.

\smallskip 
 
 A variety of characterizations for a $\Q$-unitary are obtained in Section \ref{sec_Q_uni}. In particular, we prove that a commuting quadruple $(N_0, N_1, N_2, N_3)$ is an $\Q$-unitary if and only if $(N_0, N_1)$ is a $\B_2$-unitary and $(N_1, N_2, N_3)$ is an $\E$-unitary. 
 We further show that $\Pe$-unitaries, $\B_2$-unitaries, $\E$-unitaries, $\Gamma$-unitaries and commuting pairs of unitaries can be characterized in terms of $\Q$-unitaries, thereby representing these classes into a common operator-theoretic framework. In Section \ref{sec_Q_iso}, we turn our attention to $\mathbb{H}$-isometries. Various independent characterizations for a $\Q$-isometry are given, some of which are in terms of $\B_2$-isometries and $\E$-isometries. Further, we prove that every $\mathbb{H}$-isometry admits a Wold type decomposition into two orthogonal components, namely, an $\mathbb{H}$-unitary part and a pure $\mathbb{H}$-isometry. This is analogous to the classical Wold decomposition of an isometry into a unitary part and a pure isometry part. Wold type decomposition suggests that the study of $\mathbb{H}$-isometries reduces to the class of pure $\Q$-isometries. In this direction, we recall from \cite{PalI} that every pure $\E$-isometry $(V_1, V_2, V_3)$ admits a model as a commuting triple of Toeplitz operators
 $
 (T_{F_1^*+F_2z},\, T_{F_2^*+F_1z},\, T_z)
 $
 on the vector valued Hardy space $H^2(\mathcal{D}_{V_3^*})$, where $F_1$ and $F_2$ are the fundamental operators of the $\E$-contraction $(V_1^*, V_2^*, V_3^*)$. Motivated by this model, and assuming that $F_1$ and $F_2$ are commuting normal operators satisfying
 $\|F_1^* + F_2 z\|_{\infty, \mathbb{D}} \le 1$,
 we prove in Theorem \ref{thm_fo_normalI} the existence of operators $G_0$ and $G_1$ on $\mathcal{D}_{V_3^*}$ such that the quadruple
 $
 (T_{G_0+G_1z},\, T_{F_1^*+F_2z},\, T_{F_2^*+F_1z},\, T_z)
 $
 is a pure $\Q$-isometry. In particular, we show that the fundamental operators $F_1, F_2$ corresponding to a normal $\E$-contraction $(X_1, X_2, X_3)$ are commuting normal operators and $\|F_1^*+F_2z\|_{\infty, \DC} \leq 1$.
 
 \smallskip  
 
For a bounded domain  $\Omega \subseteq \C^n$  such that $\overline{\Omega}$ is a polynomially convex compact set, let $(A_1, \dotsc, A_n)$ be an $\Omega$-contraction acting on a Hilbert space $\mathcal{H}$. An $\Omega$-isometry (or, $\Omega$-unitary) $(V_1, \dotsc, V_n)$ on a Hilbert space $\mathcal{K} \supseteq \mathcal{H}$ is said to be an \textit{$\Omega$-isometric dilation} (or, $\Omega$-\textit{unitary dilation}) of $(A_1, \dotsc, A_n)$ if 
\[
A_1^{\alpha_1}\dotsc A_n^{\alpha_n}=P_\mathcal{H}V_1^{\alpha_1} \dotsc V_n^{\alpha_n}|_\mathcal{H}
\] 
for all $\alpha_1, \dotsc, \alpha_n \in \N \cup \{0\}$. Moreover, an $\Omega$-isometric dilation is called \textit{minimal} if 
\[
\mathcal{K}=\overline{\text{span}}\left\{V_1^{\alpha_1} \dotsc V_n^{\alpha_n}h : \ h \in \mathcal{H} \ \text{and} \ \alpha_1, \dotsc, \alpha_n \in \mathbb{N} \cup \{0\} \right\}.
\]
The problem of finding an $\Omega$-isometric dilation (or equivalently, $\Omega$-unitary dilation) of an $\Omega$-contraction is simply referred to as the rational dilation problem on $\Omega$. The rational dilation succeeds on the bidisc $\D^2$ (see \cite{Ando, Nagy}) and on the symmetrized bidisc $\G_2$ (see \cite{AglerII, Bhattacharyya}). In contrast, conditional dilation results have been established for the tetrablock in \cite{Bhattacharyya-01, Pal-IV} and for the pentablock in \cite{PalN}. In Section \ref{sec_dilation}, we achieve a conditional dilation for a subclass of $\Q$-contractions whose last component is a pure contraction. Our approach builds upon the model theory of pure contractions and the dilation framework developed in \cite{Pal-IV} for a certain class of $\E$-contractions. Furthermore, we obtain a functional model for this particular class of $\Q$-contractions in terms of commuting Toeplitz operators acting on a vector-valued Hardy space. Also, we provide an equivalent criterion for the rational dilation problem associated with $\B_2, \E$ and $\Pe$ in terms of the rational dilation problem on the hexablock. In the classical setting, every contraction on a Hilbert space admits an orthogonal decomposition into a unitary part and a completely non-unitary part, where a contraction is said to be completely non-unitary if it has no nontrivial reducing subspace on which it acts as a unitary. This is known as the canonical decomposition of a contraction, e.g., see \cite{Nagy}. In Section \ref{sec_06}, we establish an analogous canonical decomposition for $\mathbb{H}$-contractions. The paper is unified by the theme of understanding how $\mathbb{H}$-contractions relate to contractions associated with the domains $\D^2, \B_2, \G_2, \E$ and $\Pe$.
	
\section{Connection of $\mathbb{H}$-contractions with $\B_2$-contractions, $\Gamma$-contractions, $\E$-contractions, $\Pe$-contractions and commuting pair of contractions}\label{sec_03}
	
	\vspace{0.2cm}
	
	\noindent We begin with remembering the definitions of the domains hexablock, tetrablock, pentablock, symmetrized bidisc from Equations (\ref{eqn_101}), (\ref{eqn:Intro-new-01}) and (\ref{eqn:Intro-new-02}). However, we recall that the hexablock is a polynomially convex domain in $\C^4$, given by 
	\begin{align*}
	&\Q=\left\{(a, x_1, x_2, x_3) \in \C \times \E :  \sup_{z_1, z_2 \in \D}|\Psi_{z_1, z_2}(a, x_1, x_2, x_3)|<1 \right\}, \text{where} \\
	& \Psi_{z_1, z_2}(a, x_1, x_2, x_3)=\frac{a\sqrt{(1-|z_1|^2)(1-|z_2|^2)}}{1-x_1z_1-x_2z_2+x_3z_1z_2}  \quad \text{for} \ (z_1, z_2) \in \DC^2 \ \text{and} \ (x_1, x_2, x_3) \in \EC.
	\end{align*}
Furthermore, the closure of $\Q$ (see Theorem 6.3 in \cite{Biswas}) is given by 
	\begin{equation}\label{closedhexa}
	\CQ=\left\{(a, x_1, x_2, x_3) \in \C \times \overline{\E} \ : \ \left|\Psi_{z_1, z_2}(a, x_1, x_2, x_3)\right| \leq 1 \ \text{for every} \ z_1, z_2 \in \D \right\}.
	\end{equation}
As mentioned earlier, an $\mathbb{H}$-contraction is a commuting quadruple of Hilbert space operators with  $\overline{\mathbb{H}}$ as a spectral set. In a similar manner, we define contractions associated with the pentablock $\Pe$, tetrablock $\E$, the Euclidean unit ball $\B_2$ in $\C^2$ and the symmetrized bidisc $\G_2$ (see Definition \ref{defn_contr}). In this section, we explore the connection of $\mathbb{H}$-contractions with $\B_2$-contractions, $\E$-contractions, $\Pe$-contractions, $\Gamma$-contractions and commuting pair of contractions. Evidently, the adjoint of an $\Q$-contraction and the restriction of an $\Q$-contraction to a joint invariant subspace are $\Q$-contractions. 
	
	\begin{lem} \label{basiclem:01}
		Let $(A,X_1, X_2, X_3)$ be an $\mathbb{H}$-contraction acting on a Hilbert space $\HS$ and let $\mathcal{L}$ be a joint invariant subspace for $A, X_1, X_2, X_3$. Then
		$ (A^*, X_1^*, X_2^*, X_3^*)$ and $(A|_{\mathcal{L}}, X_1|_{\mathcal{L}} , X_2|_{\mathcal{L}}, X_3|_{\mathcal{L}})$ are $\mathbb{H}$-contractions.
	\end{lem}	
It follows trivially from the definition of the hexablock that if $(a, x_1, x_2, x_3) \in \mathbb{H}$, then $(x_1, x_2, x_3) \in \E$. As one might expect, the same result holds for $\overline{\mathbb{H}}$ and $\overline{\E}$. Also, $\{0\} \times \E \subseteq \Q$ and $\{0\} \times \EC \subseteq \CQ$. These results naturally extend to operator-theoretic level.	
	\begin{prop}\label{prop2.3}
		If $(A, X_1, X_2, X_3)$ is an $\mathbb{H}$-contraction, then $(X_1, X_2, X_3)$ is an $\E$-contraction. Also, $(X_1, X_2, X_3)$ is an $\E$-contraction if and only if $(0, X_1, X_2, X_3)$ is an $\mathbb{H}$-contraction.
	\end{prop}
	
	\begin{proof}
		Let $(a, x_1, x_2, x_3) \in \overline{\Q}$. Then  $(x_1, x_2, x_3) \in \overline{\E}$ and so, $\overline{\Q} \subset \overline{\D} \times \overline{\E}$.  Let $g \in \C[z_1,z_2, z_3]$ and define $f(z_0, z_1, z_2, z_3)=g(z_1, z_2, z_3)$. Then
		\begin{equation*}
			\begin{split}
				\|g(X_1, X_2, X_3)\|=\|f(A, X_1, X_2, X_3)\|
				& \leq \sup\{|f(z_0, z_1, z_2, z_3)| : (z_0, z_1, z_2, z_3) \in \overline{\mathbb{H}} \}\\
				& \leq \sup\{|f(z_0, z_1, z_2, z_3)| \ : \ z_0 \in \overline{\mathbb{D}}, \ (z_1, z_2, z_3) \in \overline{\E}\}\\
				&=\sup\{|g(z_1, z_2, z_3)| : (z_1, z_2, z_3) \in \overline{\E}\}\\
			\end{split}
		\end{equation*}
		and so, $(X_1, X_2, X_3)$ is an $\E$-contraction. Now assume that $(X_1, X_2, X_3)$ is an $\E$-contraction. For any $f \in \C[z_0, z_1, z_2, z_3]$, we define $g(z_1, z_2, z_3)=f(0, z_1, z_2, z_3)$. Then
		\[
		\|f(0, X_1, X_2, X_3)\|=\|g(X_1, X_2, X_3)\| \leq \|g\|_{\infty, \EC}= \sup\{|f(0, z_1, z_2, z_3)| : (z_1, z_2, z_3) \in \EC \} \leq \|f\|_{\infty, \CQ},
		\]
		where in the last inequality we have used the fact that $\{0\} \times \EC \subseteq \CQ$. Therefore, $(0, X_1, X_2, X_3)$ is an $\Q$-contraction, which completes the proof.
	\end{proof}
	
	It is not difficult to see that if $(a, x_1, x_2, x_3) \in \CQ$ and $\alpha \in \DC$, then $(\alpha a, \alpha x_1, \alpha x_2, \alpha^2 x_3) \in \CQ$ and $(\alpha a, x_1, x_2, x_3) \in \CQ$. We refer to Chapter 6 in \cite{Biswas} for more details on these facts. We generalize these properties of hexablock at the operator-theoretic level.
	
	\begin{prop}\label{prop2.5}
		Let $(A, X_1, X_2, X_3)$ be an $\mathbb{H}$-contraction acting on a Hilbert space $\mathcal{H}$. Then the quadruples  $(\alpha A, \alpha X_1, \alpha X_2, \alpha^2 X_3)$ and $(\alpha A, X_1, X_2,  X_3)$ are $\Q$-contractions for every $\alpha \in \DC$.
	\end{prop} 
	
	\begin{proof}
		Let $\alpha \in \DC$.  The maps $f_\alpha, g_\alpha : \CQ \to \CQ$ given by $ f_\alpha(a, x_1, x_2, x_3)=(\alpha a, \alpha x_1, \alpha x_2, \alpha^2 x_3)$ and $g_\alpha(a, x_1, x_2, x_3)=(\alpha a, x_1, x_2, x_3)$ are both analytic. For any  $p \in \C[z_0, z_1, z_2, z_3]$, we have that 
		\begin{equation*}
			\begin{split}
				\|p(\alpha A, \alpha X_1, \alpha X_2, \alpha^2 X_3)\| =\|p\circ f_{\alpha}(A, X_1, X_2, X_3)\| 
				\leq \|p\circ f_\alpha\|_{\infty, \CQ}
				\leq \|p\|_{\infty, \CQ}  \qquad  \text{and} 
			\end{split}
		\end{equation*} 
		\begin{equation*}
			\begin{split}
				\|p(\alpha A, X_1, X_2, X_3)\| =\|p\circ g_{\alpha}(A, X_1, X_2, X_3)\| 
				\leq \|p\circ g_\alpha\|_{\infty, \CQ}
				\leq \|p\|_{\infty, \CQ},
			\end{split}
		\end{equation*} 
	which completes the proof. 
	\end{proof}
	
Next, we present an analogue of Part $(5)$ of Theorem \ref{thm_connect}	at the level of operators. 

	\begin{prop}\label{prop2.8}
		A pair $(A, B)$ of Hilbert space operators is a commuting pair of contractions if and only if $(A, 0, 0, B)$ is an $\Q$-contraction.
	\end{prop}
	
	\begin{proof}
		It is evident that $\EC \subseteq \DC^3$ and so, $\CQ \subseteq  \DC \times  \EC \subseteq  \DC^4$. Let $f(z_0, z_1, z_2, z_3)=z_0$. Suppose $(A, 0, 0, B)$ is an $\Q$-contraction. By Proposition \ref{prop2.3}, $(0, 0, B)$ is an $\E$-contraction. Since $\EC \subseteq \DC^3$, we have that the last component of an $\E$-contraction is a contraction (e.g. see \cite{Bhattacharyya-01}) and so, $\|B\| \leq 1$. Moreover, we have that $\|A\|=	\|f(A, 0, 0, B)\|  \leq   \|f\|_{\infty, \CQ} \leq \|f\|_{\infty, \DC \times \EC} \leq 1$. Conversely, let $(A, B)$ be a commuting pair of contractions. We have by Ando's inequality (see Chapter I in \cite{Nagy}) that 
		\begin{equation}\label{Ando}
			\|p(A, B)\| \leq \|p\|_{\infty, \DC^2},
		\end{equation}
		for every $p \in \C[z_1, z_2]$. It follows from Theorem \ref{thm_connect} that $\DC \times \{0\} \times \{0\} \times \DC \subseteq \CQ$. Let $f$ be a holomorphic polynomial in four variables and let $g(z, w)=f(z, 0, 0, w)$. Using \eqref{Ando}, we have  that
		\begin{equation*}
			\begin{split}
				\|f(A, 0, 0, B)\|  =\|g(A, B)\| \leq \|g\|_{\infty, \DC^2}
				=\sup \{|f(\underline{z})| : \underline{z} \in \DC \times \{0\} \times \{0\} \times \DC  \} \leq \|f\|_{\infty, \CQ}
			\end{split}
		\end{equation*}
		and so, $(A, 0, 0, B)$ is an $\mathbb{H}$-contraction.
	\end{proof}
	
	We now show interplay between the hexablock, the Euclidean unit ball $\B_2$ in $\C^2$ and the pentablock. At the level of scalars, we have the following lemma.
	
	\begin{lem}[\cite{Biswas}, Lemmas 6.5 \& 10.10]\label{lem2.10}
		Let $(a, x_1, x_2, x_3) \in \CQ$. Then $|a|^2+|x_1|^2 \leq 1, |a|^2+|x_2|^2 \leq 1$ and $(a, x_1+x_2, x_3) \in \PC$. 
	\end{lem}
	
	As expected, this result has an operator-theoretic extension, which is given below.
	
	\begin{prop}\label{prop2.11}
		Let $(A, X_1, X_2, X_3)$ be an $\Q$-contraction. Then $(A, X_1)$ and $(A, X_2)$ are $\B_2$-contractions. Moreover, $(A, X_1+X_2, X_3)$ is a $\Pe$-contraction.
	\end{prop}
	
	\begin{proof}
		We have by Lemma \ref{lem2.10} that the maps $f_1, f_2: \CQ \to \BC$ given by $f_1(a, x_1, x_2, x_3)=(a, x_1)$ and $f_2(a, x_1, x_2, x_3)=(a, x_2)$ are analytic. Let $g\in \C[z_1, z_2]$ and let $i \in \{1, 2\}$. Then 
		\begin{equation*}
			\begin{split}
				\|g(A, X_i)\| =\|g\circ f_i(A, X_1, X_2, X_3)\| \leq  \|g\circ f_i\|_{\infty, \CQ} \leq \sup\left\{|g(a, x_i)| : (a, x_i) \in \BC \right\}= \|g\|_{\infty, \BC} \ .	
			\end{split}
		\end{equation*} 
		Therefore, $(A, X_1), (A, X_2)$ are $\B_2$-contractions. We also have by Lemma \ref{lem2.10} that the map $f:\CQ \to \PC$ given by $f(a, x_1, x_2, x_3)=(a, x_1+x_2, x_3)$ is holomorphic. Let $g \in \C[z_1, z_2, z_3]$. Then
		\begin{equation*}
			\begin{split}
				\|g(A, X_1+X_2, X_3)\|=\|g\circ f(A, X_1, X_2, X_3)\| 
				& \leq \|g\circ f\|_{\infty, \CQ} \\
				&\leq \sup \{|g(a, x_1+x_2, x_3)|: (a, x_1, x_2, x_3) \in \CQ \}\\ 
				& \leq \|g\|_{\infty, \PC}
			\end{split}
		\end{equation*}
and so, $(A, X_1+X_2, X_3)$ is a $\Pe$-contraction. The proof is now complete.
	\end{proof}
	
	Combining Propositions \ref{prop2.3} \& \ref{prop2.11}, we have the following theorem.
	
	\begin{thm}\label{thm_310}
		If $(A, X_1, X_2, X_3)$ is an $\Q$-contraction, then
		\begin{enumerate}
			\item $(A, X_1)$ and $(A, X_2)$ are $\B_2$-contractions;
			\item $(X_1, X_2, X_3)$ is an $\E$-contraction;
			\item $(A, X_1+X_2, X_3)$ is a $\Pe$-contraction.
		\end{enumerate}
	\end{thm}
	
	The converse of Theorem \ref{thm_310} is not true. Indeed, below we show that there exist $a, x_1, x_2$ and $x_3$ in $\DC$ such that $(a, x_1), (a, x_2) \in \BC , (x_1, x_2, x_3) \in \EC, (a, x_1+x_2, x_3)\in \PC$ but $(a, x_1, x_2, x_3) \notin \CQ$.  To do so, we need the following results from \cite{Abouhajar, Biswas}.
	
		\begin{thm}[\cite{Abouhajar}, Theorem 2.2]\label{opentetrablock}
		A point $(x_1, x_2, x_3) \in \C^3$ belongs to $\E$ if and only if
		\[
		|x_1|< 1 \quad \text{and} \quad 1+|x_1|^2-|x_2|^2-|x_3|^2-2|x_1-\overline{x}_2x_3| > 0.
		\]
	\end{thm}
	
	\begin{thm}[\cite{Biswas}, Lemma 3.7 \& Theorem 3.8]\label{thm_hexa_1}
		Let $(a, x_1, x_2, x_3) \in \C \times \E$ and $(y_1, y_2, y_3) \in b\E$ with $x_1x_2=x_3$ and $y_1 \in \D$. Then 
		\[
		\sup_{z_1, z_2 \in \D}|\Psi_{z_1, z_2}(a, x_1, x_2, x_3)|=\frac{|a|}{\sqrt{(1-|x_1|^2)(1-|x_2|^2)}}  \ \text{and} \ 	\sup_{z_1, z_2 \in \D}|\Psi_{z_1, z_2}(a, y_1, y_2, y_3)|=\frac{|a|}{\sqrt{1-|y_1|^2}}.
		\]
	\end{thm}
	
We now present our example showing that the converse of Theorem \ref{thm_310} does not hold.
	
	\begin{eg}
		Let $\displaystyle \alpha_r=\frac{1}{2}\left[\sqrt{1+r^4}+(1-r^2)\right]$ and let $\beta_r=1-r^2$ for $0<r< 1\slash 2$. Define 
		\[
		a=\frac{1}{2}(\alpha_r+\beta_r)=\frac{1}{4}\left[3(1-r^2)+\sqrt{1+r^4}\right] \quad \text{and} \quad (x_1, x_2, x_3)=(r, ir, ir^2).
		\]	 
		Clearly, $\beta_r<a<\alpha_r$. It is not difficult to see that
		\[
		(1+r^4)(8-3r^2)-8=r^2(8r^2-3r^4-3)\leq r^2(8r^2-3) <0 \qquad [\text{since} \ 0<r< 1\slash 2],
		\]
		and so,
		\begin{align*}
			a^2+r^2=\frac{1}{8}[5(1+r^4)-r^2+3(1-r^2)\sqrt{1+r^4}]
			& \leq \frac{1}{8}\left[5(1+r^4)+3(1-r^2)(1+r^4)\right]\\
			&=\frac{1}{8}(1+r^4)(8-3r^2)\\
			&<1.
		\end{align*}
		Thus, $(a, x_1), (a, x_2) \in \BC$. By Theorem \ref{opentetrablock}, $(x_1, x_2, x_3) \in \E$. Let $(\lambda_1, \lambda_2)=(r, ir)$ which is a point in $\D^2$. Then $(a, x_1+x_2, x_3)=(a, \lambda_1+\lambda_2, \lambda_1\lambda_2)$ and so, $(x_1+x_2, x_3) \in \G_2$. Moreover, we have
		\[
		\frac{1}{2}|1-\overline{\lambda}_2\lambda_1|+\frac{1}{2}(1-|\lambda_1|^2)^{\frac{1}{2}}(1-|\lambda_2|^2)^{\frac{1}{2}}=	\frac{1}{2}|1+ir^2|+\frac{1}{2}(1-r^2)=\alpha_r>a.
		\]
		We recall from Theorem 5.3 in \cite{AglerIV} that a point $(b, s, p) \in \PC$ if and only if $(s, p) \in \Gamma$ and 
		\[
	|b| \leq 	\frac{1}{2}|1-\overline{\xi}_2\xi_1|+\frac{1}{2}\sqrt{(1-|\xi_1|^2)(1-|\xi_2|^2)},
		\]
	where $(s, p)=(\xi_1+\xi_2, \xi_1\xi_2)$ for $\xi_1, \xi_2 \in \DC$. Consequently, $(a, x_1+x_2, x_3) \in \PC$. We also have that
	\[
	a+r^2=\frac{1}{4}\left[3(1-r^2)+\sqrt{1+r^4}\right]+r^2=\frac{1}{4}\left[3+r^2+\sqrt{1+r^4}\right]>\frac{1}{4}(3+1)=1.
	\]
	Since $(a, x_1, x_2, x_3) \in \C \times \E$ and $x_1x_2=x_3$, we have by Theorem \ref{thm_hexa_1} that 
		\[
		\sup_{z_1, z_2 \in \D} \bigg|\frac{a\sqrt{(1-|z_1|^2)(1-|z_2|^2)}}{1-x_1z_1-x_2z_2+x_3z_1z_2}\bigg|=\frac{|a|}{\sqrt{(1-|x_1|^2)(1-|x_2|^2)}}=\frac{a}{1-r^2}>1.
		\]
		It follows from \eqref{closedhexa} that $(a, x_1, x_2, x_3) \notin \CQ$.  \qed 
	\end{eg}
	
	Given $(a, x_1), (a, x_2) \in \BC$, it is natural to ask if there exists $x_3 \in \DC$ such that $(a, x_1, x_2, x_3) \in \CQ$. Indeed, the next few results guarantee the existence of such a point $x_3 \in \DC$ .

	\begin{lem} \label{basiclem:03}
		Let $(a, x_1) \in \BC$. Then there exists $x_3 \in \mathbb{T}$ such that $(a, x_1, x_1, x_3) \in \CQ$. 
	\end{lem}
	
	\begin{proof} We have by Lemma 3.17 in \cite{PalN} that $(a, s\slash 2) \in \BC$ if and only if there exists $p \in \mathbb{T}$ such that $(a, s, p) \in \PC$. Let $(a, x_1)\in \BC$. Then there exists $x_3 \in \mathbb{T}$ such that $(a, 2x_1, p) \in \PC$ and so, we have by Theorem \ref{thm_connect} that $(a, x_1, x_1, x_3) \in \CQ$. 
	\end{proof}
	
	Next, we present a generalization of the above result. 
	
	\begin{prop}
		Let	$(a, x_1), (a, x_2) \in \BC$. Then there exists $x_3 \in \DC$ such that $(a, x_1, x_2, x_3) \in \CQ$.
	\end{prop}

	\begin{proof} Since $|x_1|, |x_2| \leq 1$, it follows from Theorem \ref{opentetrablock} that $(rx_1, rx_2, r^2x_1x_2) \in \E$ for $0<r<1$ and so, $(x_1, x_2, x_1x_2) \in \EC$. If $a=0$, then we have by Theorem \ref{thm_connect} that $(a, x_1, x_2, x_1x_2) \in \CQ$. If $a \ne 0$, then $|x_1|, |x_2| \leq \sqrt{1-|a|^2}$ and so, $x_1, x_2 \in \D$. We discuss two cases from here onwards. 
		
		\medskip 	
		\noindent \text{Case 1}. Let $|x_1|=|x_2|$. We recall from Theorem 7.1 in \cite{Abouhajar} that  $(x_1, x_2, x_3) \in b\E$ if and only if $x_1=\overline{x}_2x_3, |x_3|=1$ and $|x_2| \leq 1$.  In this case, one can find $x_3 \in \T$ such that $x_1=\overline{x}_2x_3$ and so, $(x_1, x_2, x_3) \in b\E$. It follows from Theorem \ref{thm_hexa_1} that  
		\[
		\sup_{z_1, z_2 \in \D}|\Psi_{z_1, z_2}(a, x_1, x_2, x_3)|=\frac{|a|}{\sqrt{1-|x_1|^2}} \leq 1
		\]
		and so, by \eqref{closedhexa}, $(a, x_1, x_2, x_3) \in \CQ$. 
		
		\medskip 	
		\noindent \text{Case 2}. Let $|x_1|<|x_2|$ and define $x_3=x_1\slash \overline{x}_2$. Note that $(x_1, x_2, x_3) \in \D^3$ and $x_1=\overline{x}_2x_3$. We have by Theorem \ref{opentetrablock} that $(x_1, x_2, x_3) \in \E$. We have by Proposition 3.1 in \cite{Biswas} that $\underset{z_1, z_2 \in \D}{\sup}|\Psi_{z_1, z_2}(a, x_1, x_2, x_3)|$ is attained at $(z_1, z_2)=(0, \overline{x}_2)$. Consequently, we have  
		\[
		\sup_{z_1, z_2 \in \D}|\Psi_{z_1, z_2}(a, x_1, x_2, x_3)|= \bigg|\frac{a\sqrt{1-|x_2|^2}}{1-x_2\overline{x}_2}\bigg|=\frac{|a|}{\sqrt{1-|x_2|^2}}\leq 1.
		\]
		By \eqref{closedhexa}, $(a, x_1, x_2, x_3) \in \CQ$. For the remaining case when $|x_2|<|x_1|$, the above discussion shows that $(a, x_2, x_1, x_3) \in \CQ$. It now follows from \eqref{closedhexa} that $(a, x_1, x_2, x_3) \in \CQ$. 
	\end{proof}

	We now present another main result of this section.	
	
	\begin{thm}\label{thm_connectII}
		Let $A, X_1, X_2, X_3, S$ and $P$ be operators on a Hilbert space $\HS$. Then the following holds.
		\begin{enumerate}
			\item $(A, X_1)$ is a $\B_2$-contraction if and only if $(A, X_1, 0, 0)$ is an $\Q$-contraction.
			\item $(X_1, X_2, X_3)$ is an $\E$-contraction if and only if $(0, X_1, X_2, X_3)$ is an $\Q$-contraction.
			\item $(A, S, P)$ is a $\Pe$-contraction if and only if $(A, S\slash 2, S\slash 2, P)$ is an $\Q$-contraction.
			\item $(S, P)$ is a $\Gamma$-contraction if and only if $(0, S\slash 2, S\slash 2, P)$ is an $\Q$-contraction.
			
			\item $(A, X_3)$ is a commuting pair of contractions if and only if $(A, 0, 0, X_3)$ is an $\Q$-contraction.
		\end{enumerate}
	\end{thm}

	\begin{proof} We prove the parts (1), (3) and (4). Also, parts (2) and (5) follow from Proposition \ref{prop2.3} and Proposition \ref{prop2.8} respectively.
		
		\smallskip 
		
		\noindent Proof of (1). Let $(A, X_1)$ be a $\B_2$-contraction. Take $f \in \C[z_0, z_1,z_2,z_3]$ and define $g(z_0, z_1)=f(z_0, z_1, 0, 0)$. Then 
		\begin{equation*}
			\begin{split}
				\|f(A, X_1, 0, 0)\| =\|g(A, X_1)\| \leq \|g\|_{\infty, \BC}
				=\sup \{|f(z_0, z_1, 0, 0)|: (z_0, z_1) \in \BC \} \leq \|f\|_{\infty, \CQ} ,
			\end{split}
		\end{equation*}	
		where the last inequality follows from Theorem \ref{thm_connect}. Thus, $(A, X, 0, 0)$ is an $\Q$-contraction. The converse to part (1) follows from Proposition \ref{prop2.11}.
		
		\smallskip 
		
		\noindent Proof of (3). Let $(A, S, P)$ be a $\Pe$-contraction. Take $f \in \C[z_0, z_1,z_2,z_3]$ and define $g(z_0, z_1, z_3)=f(z_0, z_1\slash 2, z_1\slash 2, z_3)$. Then 
		\begin{equation*}
			\begin{split}
				\|f(A, S\slash 2, S\slash 2, P)\| =\|g(A, S, P)\|
				\leq \|g\|_{\infty, \PC}
				&=\sup \{|f(z_0, z_1\slash 2, z_1\slash 2, z_3)|: (z_0, z_1, z_3) \in \PC \}\\
				& \leq \|f\|_{\infty, \CQ} ,
			\end{split}
		\end{equation*}	
		where the last inequality follows from Theorem \ref{thm_connect}. Therefore, $(A, S\slash 2, S\slash 2, P)$ is an $\Q$-contraction. The converse follows directly from part (3) of Theorem \ref{thm_310}. 
		\smallskip 
		
		\noindent Proof of (4). We have by Proposition 3.11 in \cite{PalN} that $(S, P)$ is a $\Gamma$-contraction if and only if $(0, S, P)$ is a $\Pe$-contraction. The latter is possible if and only if $(0, S\slash 2, S\slash 2, P)$ is an $\Q$-contraction which follows from part (3). The proof is now complete. 
	\end{proof}

	\section{The hexablock unitaries}\label{sec_Q_uni}
	
	\vspace{0.2cm}
	
	\noindent Recall that an $\mathbb{H}$-unitary is a normal $\mathbb{H}$-contraction whose joint spectrum lies in the distinguished boundary $b\mathbb{H}$ of the hexablock. In this section, we find several characterizations for the $\Q$-unitaries and find their interplay with $\B_2$-unitaries and $\E$-unitaries. We recall from the literature various characterizations for the points in $b\Q$.
\begin{thm}[\cite{Biswas}, Theorem 8.21] \label{thm:distP}
		For $(a, x_1, x_2, x_3) \in \mathbb{C}^4$, the following are equivalent:
		\begin{enumerate}
			\item $(a, x_1, x_2, x_3) \in b\mathbb{H}$;
			\item $(x_1, x_2, x_3) \in b\E, |a|^2+|x_1|^2=1$;
			\item there is a unitary matrix $U=[u_{ij}]_{2 \times 2}$ such that 
			$(a, x_1, x_2, x_3)=(u_{21}, u_{11}, u_{22}, det(U)).
			$
		\end{enumerate}
\end{thm} 
As it turns out that each of the characterizations in Theorem \ref{thm:distP} for a point in $b\Q$ gives a description of an $\Q$-unitary. In particular, we provide a characterization of $\Q$-unitaries in terms of $\B_2$-unitaries and $\E$-unitaries. For this purpose, we recall from the literature the descriptions of unitaries associated with the tetrablock $\E$ and the Euclidean unit ball $\B_n$ in $\C^n$. 
	
\begin{thm}[\cite{Bhattacharyya-01}, Theorem 5.4]\label{thm_E_unitary}
	Let $(X_1, X_2, X_3)$ be a commuting triple of operators acting on a Hilbert space $\mathcal{H}$. Then the following statements are equivalent.
	
	\begin{enumerate}
		\item $(X_1, X_2, X_3)$ is an $\E$-unitary;
		\item $(X_1, X_2, X_3)$ is an $\E$-contraction and $X_3$ is a unitary.
		\item $X_3$ is a unitary, $X_1=X_2^*X_3$ and $\|X_2\| \leq 1$;
		
	\end{enumerate}
\end{thm}	

	\begin{thm}[\cite{EschmeierII}, Section 0 \& \cite{AthavaleIII}, Proposition 2]\label{prop3.2}
	A commuting tuple $\underline{U}=(U_1, \dotsc, U_n)$ of operators is a $\B_n$-unitary if and only if each $U_j$ is normal and $U_1^*U_1+\dotsc +U_n^*U_n=I$. Moreover, $\underline{U}$ is a $\B_n$-isometry if and only if $U_1^*U_1+\dotsc +U_n^*U_n=I$.
\end{thm}

Before stating the main result of this section, we recall a basic lemma (see \cite{Nagy}, Chapter I) that is used repeatedly in this article.
	
	\begin{lem}\label{lem_basic}
		Let $A, B, T$ be operators on a Hilbert space $\HS$ and let $\|T\| \leq 1$. If $AD_T^n=D_T^nB$ for some $n \in \N$, then $AD_T=D_TB$.
	\end{lem}
	
	\begin{proof}
		We prove this result here for the sake of completeness. Let $AD_T^n=D_T^nB$ for some $n \in \N$. Then $Ap(D_T^n)=p(D_T^n)B$ for all $p \in \C[z]$. Let $p_k(z)$ be a sequence of polynomials converging uniformly to $z^{1\slash n}$ on  $[0, 1]$. Then the sequence $p_k(D_T^n)$ converges to $D_T$ in the operator norm topology. Consequently, we have that 
		$
		AD_T= \underset{k\to \infty}{\lim}Ap_k(D_T^n)=\underset{k\to \infty}{\lim} p_k(D_T^n)B=D_TB.
		$
	\end{proof}

We now present the characterizations of the unitaries associated with the hexablock. 
	
	\begin{thm}\label{thm_Q_unitary}
		Let $\underline{N}=(N_0, N_1, N_2, N_3)$ be a commuting quadruple of operators. Then the following are equivalent.
		\begin{enumerate}
			\item $\underline{N}$ is an $\mathbb{H}$-unitary ;
			\item $N_0^*$ is subnormal, $(N_1, N_2, N_3)$ is an $\E$-unitary and $N_0^*N_0=I-N_1^*N_1$ ;
			
			\item $(N_1, N_2, N_3)$ is an $\E$-unitary, $N_0^*N_0=I-N_1^*N_1$ and $N_0N_0^*=I-N_1N_1^*$;
			
			\item $(N_0, N_1)$ is a $\B_2$-unitary and $(N_1, N_2, N_3)$ is an $\E$-unitary ;
			
			\item there is a $2 \times 2$ unitary block matrix $A=[A_{ij}]$, where $A_{ij}$ are commuting normal operators, such that 	$
			\underline{N}=(A_{21}, A_{11}, A_{22}, A_{11}A_{22}-A_{12}A_{21}).
			$
		\end{enumerate}
	\end{thm}
	
	\begin{proof} The condition $(3) \implies (2)$ is trivial.  Suppose condition $(2)$ holds. Then $(N_0, N_1)$ is a commuting pair of operators such that $N_0^*$ is subnormal and $N_1$ is normal. By Fuglede's theorem \cite{Fuglede}, $N_0$ doubly commutes with $N_1$. Thus,  $N_0$ is quasinormal, i.e., $N_0$ commutes with $N_0^*N_0$, because $N_0^*N_0=I-N_1^*N_1$. It is well-known that every quasinormal operator is subnormal and so, $N_0$ is subnormal. Since both $N_0, N_0^*$ are subnormal, it follows that $N_0$ is normal. Hence, condition $(3)$ follows, thereby giving the equivalence of conditions $(2)$ and $(3)$. The equivalence of $(3)$ and $(4)$ follows from Theorem \ref{prop3.2}.	We now prove the remaining implications.
		
		\medskip
		
		\noindent $(1) \implies (2).$ Let $\underline{N}$ be an $\Q$-unitary. Then $N_0, N_1, N_2, N_3$ are commuting normal operators such that $\sigma_T(N_0, N_1, N_2, N_3) \subseteq b\mathbb{H}$. Let $(x_1, x_2, x_3) \in \sigma_T(N_1, N_2, N_3)$. We have by spectral mapping theorem that there exists some $a \in \C$ such that $(a, x_1, x_2, x_3) \in \sigma_T(\underline{N}) \subseteq b\Q$. It follows from Theorem \ref{thm:distP} that $(x_1, x_2, x_3) \in b\E$. Thus, $\sigma_T(N_1, N_2, N_3) \subseteq b\E$ and so, $(N_1, N_2, N_3)$ is an $\E$-unitary. The commutative $C^*$-algebra generated by $N_0, N_1, N_2, N_3$ is isometrically isomorphic to the $C(\sigma_T(\underline{N}))$ via the map that takes the coordinate function $z_i$ to $N_i$ for $i=0, 1,2,3$. We have by Theorem \ref{thm:distP} that $|z_0|^2+|z_1|^2=1$ on $b\mathbb{H}$ and so, $N_0^*N_0=I-N_1^*N_1$. 
		
		\medskip
		
		\noindent $(2) \implies (1).$ It follows from the above discussion as in the proof of the implication $(2) \implies (3)$ that $N_0$ is normal. Let $(a, x_1, x_2, x_3) \in \sigma_T(\underline{N})$. By the projection property of joint spectrum, we have that $(x_1, x_2, x_3) \in \sigma_T(N_1, N_2, N_3)$. Since $(N_1, N_2, N_3)$ is an $\E$-unitary, $(x_1, x_2, x_3)\in b\E$. Moreover, the map	
		$
		f(z_1, z_2, z_3, z_4)=|z_1|^2+|z_2|^2-1
		$
		is continuous on $\sigma_T(\underline{N})$. By continuous functional calculus, we must have 
		\[
		f(N_0, N_1, N_2, N_3)=N_0^*N_0+N_1^*N_1-I=0 \quad \text{and so,} \quad \{0\}=\sigma_T(f(\underline{N}))
		=f(\sigma_T(\underline{N})),			\]
		where the last equality follows from the spectral mapping principle. Therefore, $|a|^2+|x_1|^2=1$ and so, $\sigma_T(\underline{N}) \subseteq b\Q$. Hence, $\underline{N}$ is an  $\Q$-unitary. 
		
		\medskip		
		
		\noindent $(2) \implies (5).$ Define
		$
		A=[A_{ij}]_{i, j=1}^2=\begin{bmatrix}
			N_1 & -N_0^*N_3\\
			N_0 & N_2
		\end{bmatrix}.
		$
		Since $N_0, N_1, N_2, N_3$ are commuting normal operators, we have that $A_{ij}$ are commuting normal operators. We are also given that $(N_1, N_2, N_3)$ is an $\E$-unitary. By Theorem \ref{thm_E_unitary}, $N_3$ is a unitary, $N_1=N_2^*N_3$ and so, $N_2^*N_2=N_1^*N_1$. Then
		\begin{equation*}
			\begin{split}
				AA^*=\begin{bmatrix}
					N_1N_1^*+N_0^*N_3N_3^*N_0 & N_1N_0^*-N_0^*N_3N_2^* \\ \\
					N_0N_1^*-N_2N_3^*N_0 & N_0N_0^*+N_2N_2^* 
				\end{bmatrix}=\begin{bmatrix}
					N_1N_1^*+N_0^*N_0 & N_0^*(N_1-N_2^*N_3) \\ \\
					N_0(N_1^*-N_2N_3^*) & N_0^*N_0+N_1^*N_1 
				\end{bmatrix}=\begin{bmatrix}
					I & 0 \\
					0 & I\\
				\end{bmatrix}
			\end{split}
		\end{equation*}
		and 	
		\begin{equation*}
			\begin{split}
				A^*A=\begin{bmatrix}
					N_1^*N_1+N_0^*N_0 & -N_1^*N_0^*N_3+N_0^*N_2 \\ \\
					-N_3^*N_0N_1+N_2^*N_0 & N_3^*N_0N_0^*N_3+N_2^*N_2 
				\end{bmatrix}=\begin{bmatrix}
					N_1N_1^*+N_0^*N_0 & N_0^*(N_2-N_1^*N_3) \\ \\
					N_0(N_2^*-N_1N_3^*) & N_0^*N_0+N_1^*N_1 
				\end{bmatrix}=\begin{bmatrix}
					I & 0 \\
					0 & I\\
				\end{bmatrix}. 
			\end{split}
		\end{equation*}
		Evidently, $(A_{21}, A_{11}, A_{22})=(N_0, N_1, N_2)$ and we have that
		\[
		A_{11}A_{22}-A_{12}A_{21}=N_1N_2+N_0^*N_0N_3=N_2^*N_2N_3+N_0^*N_0N_3=\left(N_2^*N_2+N_0^*N_0\right)N_3=N_3.
		\]  
		Thus, $\underline{N}=(A_{21}, A_{11}, A_{22}, A_{11}A_{22}-A_{12}A_{21})$ and so, $(5)$ holds.
		
		\medskip
		
		\noindent $(5) \implies (2).$ Assume that $A$ is a unitary block matrix $[A_{ij}]_{i, j=1}^2$, where $A_{ij}$ are commuting normal operators such that $\underline{N}=(A_{21}, A_{11}, A_{22}, A_{11}A_{22}-A_{12}A_{21})$. Then $\|N_j\| \leq \|A_{jj}\| \leq 1$ for $j=1, 2$. The condition  $A^*A=I$ gives the following set of equations.
		\begin{equation}\label{eq1}
			A_{11}^*A_{11}+A_{21}^*A_{21}=I, \quad
			A_{12}^*A_{12}+A_{22}^*A_{22}=I, 
		\end{equation}
		\begin{equation}\label{eq2}
			A_{12}^*A_{11}+A_{22}^*A_{21}=0, \quad A_{11}^*A_{12}+A_{21}^*A_{22}=0.
		\end{equation}
		Again, $AA^*=I$ provides the following equations.
		\begin{equation}\label{eq3}
			A_{11}A_{11}^*+A_{12}A_{12}^*=I, \quad A_{21}A_{21}^*+A_{22}A_{22}^*=I,
		\end{equation}
		\begin{equation}\label{eq4}
			A_{21}A_{11}^*+A_{22}A_{12}^*=0, \quad 
			A_{11}A_{21}^*+A_{12}A_{22}^*=0 .
		\end{equation}
		Using the above equations, we have the following.
		\begin{equation*}
			\begin{split}
				N_1^*N_3
				=A_{11}^*(A_{11}A_{22}-A_{12}A_{21})
				&=(A_{11}^*A_{11})A_{22}-A_{11}^*A_{12}A_{21} \\
				& =(I-A_{21}^*A_{21})A_{22}-A_{11}^*A_{12}A_{21}  \quad [\text{ by } {(\ref{eq1})}] \ \\
				&=A_{22}-A_{22}A_{21}^*A_{21} -(A_{11}^*A_{12})A_{21}\\
				&=A_{22}-A_{22}A_{21}^*A_{21} +(A_{22}A_{21}^*)A_{21} \ \quad [\text{ by } {(\ref{eq2})}]\\
				&=N_2.\\
			\end{split}
		\end{equation*}
		To show that $N_3$ is unitary, we just need to check that $N_3^*N_3= I$, because $N_3$ is a normal operator.
		\begin{equation*}
			\begin{split}
				N_3^*N_3
				&=(A_{11}^*A_{22}^*-A_{12}^*A_{21}^*)(A_{11}A_{22}-A_{12}A_{21})\\
				&=A_{11}^*A_{11}A_{22}^*A_{22}-(A_{11}^*A_{12})A_{22}^*A_{21}-(A_{12}^*A_{11})A_{21}^*A_{22}+A_{12}^*A_{21}^*A_{12}A_{21}\\
				& =A_{11}^*A_{11}A_{22}^*A_{22}+(A_{21}^*A_{22})A_{22}^*A_{21}+(A_{22}^*A_{21})A_{21}^*A_{22}+A_{12}^*A_{21}^*A_{12}A_{21} \quad [\text{ by } (\ref{eq2})]\\
				&=(A_{11}^*A_{11}+A_{21}^*A_{21})A_{22}^*A_{22}+(A_{22}^*A_{22}+A_{12}^*A_{12})A_{21}^*A_{21}\\
				& =A_{22}^*A_{22}+A_{21}^*A_{21} \quad [\text{ by } (\ref{eq1})]\\
				& =I. \quad [\text{ by } (\ref{eq3})]
			\end{split}
		\end{equation*}
		Hence, $N_1, N_2, N_3$ are commuting  normal operators satisfying $\|N_2\| \leq 1, N_1^*N_3=N_2$ and $N_3^*N_3=N_3^*N_3=I$. It follows from Theorem \ref{thm_E_unitary} that $(N_1, N_2,N_3)$ is an $\E$-unitary. Moreover, we have by \eqref{eq1} that 
		$
		N_0^*N_0+N_1^*N_1=A_{21}^*A_{21}+A_{11}^*A_{11}=I.
		$
		The proof is now complete.
	\end{proof}
	
	\smallskip
	
	We show that the hypothesis of subnormality of $N_0^*$ in part (2) of Theorem \ref{thm_Q_unitary} cannot be dropped.
	
	\begin{eg} \label{exm:imp}
		Let $T_z$ be the unilateral shift on $\ell^2(\mathbb{N})$. Define $\underline{N}=(N_0, N_1, N_2, N_3)=(T_z, 0, 0, I)$ on $\ell^2(\mathbb{N})$. It is not difficult to see that $(0, 0, I)$ is an $\E$-unitary since $\sigma_T(0, 0, I)=\{(0, 0, 1)\} \subseteq b\E$. Moreover, $N_0^*N_0+N_1^*N_1=T_z^*T_z=I$. Thus, $\underline{N}$ satisfies the condition $(2)$ of Theorem \ref{thm_Q_unitary} except that $N_0^*$ is subnormal. Note that $(0, 0, 0, 1) \in \sigma_T(\underline{N})$ but $(0, 0, 0, 1) \notin b\Q$ by virtue of Theorem \ref{thm:distP}. Thus, $\underline{N}$ is not an $\Q$-unitary.  \qed 
	\end{eg}
	
	We also mention that Theorem \ref{thm_Q_unitary} fails if we do not assume that $(N_0, N_1)$ is a $\B_2$-unitary.

	\begin{eg}
		Let  $\theta \in \mathbb{R}$ and let
		$
		(a,x_1, x_2, x_3)=(0, 0, 0, e^{i\theta}).
		$
		Then $(0, 0, e^{i\theta}) \in b \E$. Moreover, 
		\[
		\sup_{z_1, z_2 \in \D}\left|\frac{a\sqrt{(1-|z_1|^2)(1-|z_2|^2)}}{1-x_1z_1-x_2z_2+x_3z_1z_2}\right|=0 <1
		\]
		and so, $(0, 0, 0, e^{i\theta}) \in \overline{\Q}$. However, $|a|^2+|x_1|^2 \ne 1$ and by Theorem \ref{thm:distP}, $(a, x_1, x_2, x_3) \notin b\Q$. Also, it shows that $
		b\Q \ne \{(a, x_1, x_2, x_3) \in \overline{\Q} \ : \ |x_3|=1 \}$. \qed 
	\end{eg}
	
The operator theory on the pentablock has been recently studied in \cite{JindalII, PalN}. We mention a characterization of $\Pe$-unitaries in this context. 
	
	\begin{thm}[\cite{PalN}, Theorem 5.2]\label{thm_P_uni}
	A commuting triple $\underline{N}=(N_1, N_2, N_3)$  of  operators is a $\mathbb{P}$-unitary if and only if $(N_1, N_2\slash 2)$ is a $\B_2$-unitary and $(N_2, N_3)$ is a $\Gamma$-unitary.
	\end{thm}

	The next result follows immediately from  Theorem \ref{thm_P_uni} and Theorem \ref{thm_Q_unitary}.

	\begin{cor}\label{cor_P_Q_uni}
		Let $A, X_3, S$ and $P$ be operators on a Hilbert space $\HS$. Then the following holds.
		\begin{enumerate}
			\item $(A, S, P)$ is a $\Pe$-unitary if and only if $(A, S\slash 2, S\slash 2, P)$ is an $\Q$-unitary.
			
			\item $(A, X_3)$ is a commuting pair of unitaries if and only if $(A, 0, 0, X_3)$ is an $\Q$-unitary.
		\end{enumerate}
	\end{cor}
	
	The above corollary is an analogue of parts (3) and (5) of Theorem \ref{thm_connectII} at the level of unitaries. However, the remaining parts do not find a similar analogue. For example, $(I,0)$ is a $\B_2$-unitary but $(I, 0, 0, 0)$ is not an $\Q$-unitary showing that an analogue of part (1) of Theorem \ref{thm_connectII} does not hold here. Also, $(0, 0, I)$ is an $\E$-unitary, $(0, I)$ is a $\Gamma$-isometry but $(0, 0, 0, I)$ is not an $\Q$-unitary which shows the failure of analogues of parts (2) and (4) of Theorem \ref{thm_connectII}. This gives rise to a natural question if one can characterize $\B_2$-unitaries, $\E$-unitaries and $\Gamma$-unitaries in terms of $\Q$-unitaries. It is equivalent to asking if one construct an $\mathbb{H}$-unitary from a given $\B_2$-unitary\slash $\Gamma$-unitary\slash $\E$-unitary. The rest of the section does answer this only.

	\smallskip 
	
	Before proceeding further, we recall the polar decomposition of normal operators. Let $N$  be a normal operator on a Hilbert space $\mathcal{H}$. Then there exists a unitary operator $U$ on $\mathcal{H}$ such that $N=U|N|=|N|U$ and $U$ commutes with any operator that commutes with $N$. Recall that $|N|$ denotes the operator $(N^*N)^{1\slash 2}$.   For further details, one can refer to Theorem 12.35 in \cite{Rudin}.
	
	\begin{prop}
		Let $A, X_1$ be operators on a Hilbert space $\HS$. Then the following are equivalent:
		\begin{enumerate}
			\item $(A, X_1)$ is a $\B_2$-unitary;
			\item $(A, X_1, X_1^*, I)$ is an $\Q$-unitary;
			\item $(A, X_1, |X_1|, U)$ is an $\Q$-unitary, where $X_1=|X_1|U$ is the polar decomposition of $X_1$.  
		\end{enumerate}
	\end{prop}	
	
	\begin{proof}
		The equivalence of $(1)$ and $(2)$ follows directly from Theorem \ref{thm_Q_unitary} and Theorem \ref{thm_E_unitary}. Also, $(3) \implies (1)$ follows from Theorem \ref{thm_Q_unitary}. We prove $(1) \implies (3)$. Let $(A, X_1)$ be a $\B_2$-unitary. Then $A$ and $X_1$ are commuting normal operators and so, $(A, X_1, |X_1|, U)$ is a commuting quadruple of normal contractions. By Theorem \ref{thm_E_unitary}, $(X_1, |X_1|, U)$ is an $\E$-unitary and so, by Theorem \ref{thm_Q_unitary}, $(A, X_1, |X_1|, U)$ is an $\Q$-unitary. The proof is now complete. 
	\end{proof}

	\begin{prop}\label{prop_BE_uni}
		An operator triple $(N_1, N_2, N_3)$ is an $\E$-unitary if and only if $(D_{N_1}, N_1, N_2, N_3)$ is an $\Q$-unitary.
	\end{prop}	
	
	\begin{proof}
		Let $(N_1, N_2, N_3)$ be an $\E$-unitary. By Theorem \ref{thm_E_unitary}, $\|N_1\|, \|N_2\| \leq 1$ and $N_1^*N_1=N_2^*N_2$. Since $N_1, N_2, N_3$ are commuting normal operators, we have by Fuglede's theorem \cite{Fuglede} that $N_1, N_2, N_3$ doubly commute with each other. Thus, $D_{N_1}$ commutes with $N_1, N_2$ and $N_3$. Consequently, the quadruple $(D_{N_1}, N_1, N_2, N_3)$ consists of commuting normal operators such that $(N_1, N_2, N_3)$ is an $\E$-unitary and 
		$
		D_{N_1}^*D_{N_1}+N_1^*N_1=D_{N_1}^2+N_1^*N_1=I.
		$
		We have by Theorem \ref{thm_Q_unitary} that $(D_{N_1}, N_1, N_2, N_3)$ is an $\Q$-unitary. The converse follows from Theorem \ref{thm_Q_unitary}.
	\end{proof}
	
	One can use the polar decomposition for normal operators and generalize the above proposition. This provides another characterization for an $\Q$-unitary.

	\begin{thm}
		Let $\underline{N}=(N_0, N_1, N_2, N_3)$ be a commuting quadruple of operators acting on a Hilbert space $\mathcal{H}$. Then $\underline{N}$ is an $\mathbb{H}$-unitary if and only if $(N_1, N_2, N_3)$ is an $\E$-unitary and there is a unitary $U$ on $\mathcal{H}$ such that $UN_j=N_jU$ for $j=1, 2, 3$ and 
		$
		N_0=UD_{N_1}=D_{N_1}U.
		$
	\end{thm}
	
	\begin{proof}
		Assume that $(N_1, N_2, N_3)$ is an $\E$-unitary and $U$ is a unitary on $\mathcal{H}$ such that $UN_j=N_jU$ for $j=1, 2, 3$. By Fuglede's theorem,  $UN_j^*=N_j^*U$ for $j=1, 2, 3$ and so, $UD_{N_1}^2=D_{N_1}^2U$. By Lemma \ref{lem_basic}, we have that $UD_{N_1}=D_{N_1}U$. Take $N_0=UD_{N_1}$. It is easy to see that $N_0, N_1, N_2, N_3$ are commuting normal operators and $N_0^*N_0+N_1^*N_1=I$. We have by Theorem \ref{thm_Q_unitary} that $(N_0, N_1, N_2, N_3)$ is an $\Q$-unitary. Conversely, suppose that $(N_0, N_1, N_2, N_3)$ is an $\Q$-unitary.  It follows from Theorem \ref{thm_Q_unitary} that $(N_1, N_2, N_3)$ is an $\E$-unitary and $N_0^*N_0=I-N_1^*N_1=D_{N_1}^2$.  So, $(N_0^*N_0)^{1\slash 2}=D_{N_1}$.
		The polar decomposition theorem (see the discussion after Corollary \ref{cor_P_Q_uni}) ensures the existence of a unitary $U$ on $\mathcal{H}$ such that $UN_j=N_jU$ for $j=1, 2, 3$ and 	
		$
		N_0=U(N_0^*N_0)^{1\slash 2}=(N_0^*N_0)^{1\slash 2}U.
		$
		Therefore, $N_0=UD_{N_1}=D_{N_1}U$ which completes the proof.
	\end{proof}

	\begin{cor}
		An operator $(S, P)$ is a $\Gamma$-unitary if and only if $(D_{S\slash 2}, S\slash 2, S \slash 2, P)$ is an $\Q$-unitary.
	\end{cor}
	
	\begin{proof}
		We have by Theorem 2.2 in \cite{AglerVII} that a commuting pair $(S, P)$ of normal operators is a $\Gamma$-unitary if and only if $S=S^*P$, $P$ is a unitary and $\|S\| \leq 2$. It follows from Theorem \ref{thm_E_unitary} that $(S\slash 2, S\slash 2, P)$ is an $\E$-unitary if and only if $(S, P)$ is a $\Gamma$-unitary. The desired conclusion now follows from Proposition \ref{prop_BE_uni}.
	\end{proof}

	\section{The hexablock isometries}\label{sec_Q_iso}
	
	\vspace{0.2cm}
	
	\noindent In this section, we explore the structure of an $\Q$-isometry and identify various ways to characterize them. By definition, an $\Q$-isometry is the restriction of an $\Q$-unitary $(A, X_1, X_2, X_3)$ to a joint invariant subspace for $A, X_1, X_2$ and $X_3$. Consequently, an $\Q$-isometry is a subnormal quadruple. Note that a commuting tuple of operators $\UT=(T_1, \dotsc, T_n)$ on a Hilbert space $\mathcal{H}$ is said to be \textit{subnormal} if there is a Hilbert space $\mathcal{K} \supseteq \mathcal{H}$ and a commuting tuple of normal operators $\underline{N}=(N_1, \dotsc, N_n)$ on $\mathcal{K}$ such that $\mathcal{H}$ is a joint invariant subspace for  $N_1, \dotsc, N_n$ and 
	\[
	(T_1, \dotsc, T_n)=(N_1|_{\mathcal{H}}, \dotsc, N_n|_{\mathcal{H}}).
	\]
	The tuple $\underline{N}$ is referred to as a normal extension of $\underline{T}$. It follows from the theory of subnormal operators (see \cite{Lubin, Athavale}) that every subnormal tuple has a minimal normal extension to the space 
	\[
	\overline{span}\left\{N_1^{*k_1} \dotsc N_n^{*k_n}h \ : \ h \in \mathcal{H} \ \& \ \ k_1, \dotsc, k_n \in \N \cup \{0\} \right\},
	\]
	and this minimal normal extension is unique up to a unitary equivalence. We mention a few useful results from the literature about subnormal tuples.

	\begin{lem}[\cite{Lubin}, Corollary 2]\label{Lubin2}
		Let $\underline{N}=(N_1, \dotsc, N_n)$ be the minimal normal extension of a subnormal tuple $\underline{S}=(S_1, \dotsc, S_n)$. Then $p(\underline{N})$ is unitarily equivalent to the minimal normal extension of $p(\underline{S})$ for all $p \in \C[z_1, \dotsc, z_n]$.
	\end{lem}
	
	The following result due to Athavale \cite{AthavaleIII} gives a sufficient condition for the subnormality of commuting operator tuples. This result was also proved independently by Arveson \cite{ArvesonIII}.
	
	\begin{lem}[\cite{ArvesonIII}, Corollary 1]\label{subnormal}
		Let $T_1, \dotsc, T_n$ be a set of commuting operators on a Hilbert space $\mathcal{H}$ such that $T_1^*T_1+\dotsc+T_n^*T_n=I$. Then $(T_1, \dotsc, T_n)$ is a subnormal tuple. 
	\end{lem}
We will use the above results on subnormal operators from the literature to arrive at the results in this section. Our first main result of this section is a characterization theorem for $\Q$-isometries. For this purpose, we present the following description of $\E$-isometries. 
	
		\begin{thm}[\cite{Bhattacharyya-01}, Theorem 5.7]\label{thm_E_isometry}
		Let $(V_1, V_2, V_3)$ be a commuting triple of operators acting on a Hilbert space $\mathcal{H}$. Then the following statements are equivalent:
		
		\begin{enumerate}
			\item $(V_1, V_2, V_3)$ is an $\E$-isometry;
			\item $(V_1, V_2, V_3)$ is an $\E$-contraction and $V_3$ is an isometry;		
			\item $V_3$ is an isometry, $V_1=V_2^*V_3$ and $\|V_2\| \leq 1$;
			\item $V_3$ is an isometry, the spectral radii $r(V_1) \leq 1, r(V_2) \leq 1$ and $V_1=V_2^*V_3$.
		\end{enumerate}
	\end{thm}

	\begin{thm}\label{thm_Q_isometry}
		Let $(V_0, V_1, V_2, V_3)$ be a commuting quadruple of operators acting on a Hilbert space $\mathcal{H}$. Then the following are equivalent:
		\begin{enumerate}
			\item $(V_0, V_1, V_2, V_3)$ is an $\Q$-isometry;

			\item $(V_0, V_1)$ is a $\B_2$-isometry and $(V_1, V_2, V_3)$ is an $\E$-isometry;
			
			\item 	$(V_0, V_2)$ is a $\B_2$-isometry and $(V_1, V_2, V_3)$ is an $\E$-isometry;
			
			\item $(V_0, V_1)$ and $(V_0, V_2)$ are $\B_2$-isometries and $(V_1, V_2, V_3)$ is an $\E$-isometry.

		\end{enumerate}
	\end{thm}
	
	\begin{proof}
		$(1) \implies (2).$ Let $(V_0, V_1,V_2,V_3)$ be an $\Q$-isometry on a Hilbert space $\HS$. By definition, there exists an $\Q$-unitary $(U_0, U_1, U_2, U_3)$ acting on a larger Hilbert space $\mathcal{K}$ containing $\mathcal{H}$ such that $\mathcal{H}$ is an invariant subspace for each $U_j$ and $(V_0, V_1, V_2, V_3)=(U_0|_\HS, U_1|_{\mathcal{H}}, U_2|_{\mathcal{H}} ,U_3|_{\mathcal{H}})$. We have by Theorem \ref{thm_Q_unitary} that $(U_0, U_1)$ is a $\mathbb{B}_2$-unitary and $(U_1, U_2, U_3)$ is an $\E$-unitary. Therefore, $(V_1, V_2, V_3)$ is an $\E$-isometry. Since each $V_j=U_j|_\HS$, we have that 
		$
		V_j^*V_j=P_\mathcal{H}U_j^*U_j|_\mathcal{H}$. Again by Theorem \ref{thm_Q_unitary}, we have
		\[
		I-V_0^*V_0-V_1^*V_1=P_{\mathcal{H}}\bigg(I-U_0^*U_0-U_1^*U_1\bigg)\bigg|_\mathcal{H}=0.
		\] 
		It follows from Theorem \ref{prop3.2} that $(V_0, V_1)$ is a $\B_2$-isometry.
		
		\medskip 
		
		\noindent $(2) \implies (3).$ We only need to show that $(V_0, V_2)$ is a $\B_2$-isometry. Since $(V_1, V_2, V_3)$ is an $\E$-isometry, we have by Theorem \ref{thm_E_isometry} that $V_1=V_2^*V_3$ and $V_3$ is an isometry. Then 
		\[
		V_1^*V_1=V_3^*V_2V_1=V_3^*V_1V_2=V_3^*V_2^*V_3V_2
		=V_2^*V_3^*V_3V_2=V_2^*V_2. 	
		\]		
		Since $(V_0, V_1)$ is a $\mathbb{B}_2$-isometry and $V_1^*V_1=V_2^*V_2$, we have that $V_0^*V_0+V_2^*V_2=V_0^*V_0+V_1^*V_1=I$ and by Theorem \ref{prop3.2}, $(V_0, V_2)$ is a $\B_2$-isometry.

		\medskip 
		
		\noindent $(3) \implies (4).$  One can easily employ similar techniques as above in $(2) \implies (3)$ to show that $(V_0, V_1)$ is a $\B_2$-isometry. 
		
		\medskip
		
		\noindent	$(4) \implies (1).$ We first show that $(V_0, V_1, V_2, V_3)$ has a simultaneous normal extension. Since $(V_0, V_1)$ and $(V_0, V_2)$ are $\B_2$-isometries and $(V_1, V_2, V_3)$ is an $\E$-isometry, we have by Theorems \ref{prop3.2} \& \ref{thm_E_isometry} that
		$
		V_0^*V_0+V_1^*V_1=I, V_0^*V_0+V_2^*V_2=I$ and $V_3^*V_3=I$.	Therefore, 
		\[
		2V_0^*V_0+V_1^*V_1+V_2^*V_2+V_3^*V_3=3I \quad \text{and so,} \quad \frac{2}{3}V_0^*V_0+\frac{1}{3}V_1^*V_1+\frac{1}{3}V_2^*V_2+\frac{1}{3}V_3^*V_3=I.
		\]
		We have by Lemma \ref{subnormal} that $(V_0, V_1, V_2, V_3)$ admits a simultaneous normal extension. Thus, there exist a Hilbert space $\mathcal{K}$ containing $\HS$ and a commuting quadruple $(U_0, U_1, U_2, U_3)$ of normal operators acting on $\mathcal{K}$ such that $\mathcal{H}$ is a joint invariant subspace for $U_0, U_1, U_2, U_3$ and $V_j=U_j|_\HS$ for $0 \leq j \leq 3$. Without loss of generality, we assume that $(U_0, U_1, U_2, U_3)$ on $\mathcal{K}$ is the minimal normal extension of the triple $(V_0, V_1, V_2, V_3)$. Then
		\[
		\mathcal{K}=		\overline{\mbox{span}}\{U_0^{*j_0}U_1^{*j_1}U_2^{*j_2} U_3^{*j_3}h \ | \ j_0, j_1, j_2, j_3 \geq 0, \ h \in \mathcal{H}\}.
		\]
		We prove that $(U_0, U_1, U_2, U_3)$ acting on $\mathcal{K}$ is an $\mathbb{H}$-unitary. We have by Lemma \ref{Lubin2} that $U_j$ is unitarily equivalent to the minimal normal extension of $V_j$ for $j=0, 1, 2, 3$. Consequently, $U_3$ is a unitary by being the minimal normal extension of the isometry $V_3$. Choose a holomorphic polynomial $h$ in three variables and define $f(z_0, z_1, z_2, z_3)=h(z_1, z_2, z_3)$. Evidently, $S'=f(V_0, V_1, V_2, V_3)$ is a subnormal operator and let $N'$ be its minimal normal extension. By Lemma \ref{Lubin2}, $f(U_0, U_1, U_2, U_3)$ and $N'$ are unitarily equivalent. Bram proved in \cite{Bram} that the spectral inclusion relation holds for subnormal operators, i.e., $\sigma(N') \subseteq \sigma(S')$. Using the normality of $N'$, we have that 
		\begin{equation*}
			\begin{split}
				\|N'\|=\sup\{|z| : z  \in \sigma(N')\}
				\leq \sup\{|z| : z  \in \sigma(S')\}
				& = \sup\{|z| : z \in \sigma(f(V_0, V_1, V_2, V_3))\}\\
				& = \sup\{|z| : z \in \sigma(h(V_1, V_2, V_3))\}\\
				& = \sup\{|z| : z \in h(\sigma_T(V_1, V_2, V_3))\}\\
				& \leq \sup\{|z| : z \in h(\overline{\E})\}\\
				&=\|h\|_{\infty, \overline{\E}}.
			\end{split}
		\end{equation*} 
		Since $h(U_1, U_2, U_3)=f(U_0, U_1, U_2, U_3)$ and $f(U_0, U_1, U_2, U_3)$ is unitarily equivalent to $N'$, it follows that
		$
		\|h(U_1, U_2, U_3)\| =\|N'\| \leq \|h\|_{\infty, \overline{\E}}.
		$
		Thus, $(U_1, U_2, U_3)$ is an $\E$-contraction and by Theorem \ref{thm_E_unitary}, $(U_1, U_2, U_3)$ is an $\E$-unitary. Since $(V_0, V_1)$ is a $\B_2$-isometry, we have by Theorem \ref{prop3.2} that 
		\begin{equation}\label{eqn_601}
			\|V_0y\|^2+\|V_1y\|^2-\|y\|^2=\langle (V_0^*V_0-I+V_1^*V_1)y, y \rangle=0 
		\end{equation}
		for every $y \in \mathcal{H}$. Let $x \in \mathcal{H}$. Then
		\begin{small}
			\begin{equation*}
				\begin{split}
					& \quad  \|(U_0^*U_0+U_1^*U_1-I)x\|^2 \\
					& = \left(\|U_0^2x\|^2+\|U_0U_1x\|^2-\|U_0x\|^2\right)
					+\left(\|U_0U_1x\|^2+\|U_1^2x\|^2-\|U_1x\|^2\right)
					-\left(\|U_0x\|^2+\|U_1x\|^2-\|x\|^2\right)\\	
					& =\left(\|V_0^2x\|^2+\|V_0V_1x\|^2-\|V_0x\|^2\right)
					+\left(\|V_0V_1x\|^2+\|V_1^2x\|^2-\|V_1x\|^2\right)
					-\left(\|V_0x\|^2+\|V_1x\|^2-\|x\|^2\right) \\
					& =0,	
				\end{split}
			\end{equation*}
		\end{small} 
		where the last equality follows from \eqref{eqn_601}. Thus, $(I-U_0^*U_0-U_1^*U_1)x=0$ for all $x \in \mathcal{H}$. It follows from the definition of $\mathcal{K}$ that $I-U_0^*U_0-U_1^*U_1=0$ on $\mathcal{K}$. By Theorem \ref{thm_Q_unitary}, $(U_0, U_1, U_2, U_3)$ is an $\mathbb{H}$-unitary and so, $(V_0, V_1, V_2, V_3)$ is an $\mathbb{H}$-isometry. The proof is now complete.
	\end{proof}

	The next result is an immediate consequence of Theorem \ref{thm_Q_unitary} and Theorem \ref{thm_Q_isometry}.

	\begin{cor}\label{P_unitaryII}
		Let $\underline{N}=(N_0, N_1, N_2, N_3)$ be a commuting quadruple of operators. Then $\underline{N}$ is an $\mathbb{H}$-unitary if and only if both $(N_0, N_1, N_2, N_3)$ and $(N_0^*, N_1^*, N_2^*, N_3^*)$ are $\mathbb{H}$-isometries.
	\end{cor}
	
We now present an analogue of Corollary \ref{cor_P_Q_uni}, which characterizes $\mathbb{P}$-isometries and $\mathbb{B}_2$-isometries in terms of $\mathbb{H}$-isometries. The proof of the following corollary is an immediate consequence of Theorem \ref{thm_Q_isometry}, together with the characterization of $\mathbb{P}$-isometries stated below.

\begin{thm}[\cite{PalN}, Theorem 6.7]\label{thm_P_iso}
	Let $(V_1, V_2, V_3)$ be a commuting triple of operators on a Hilbert space $\mathcal{H}$. Then $(V_1, V_2, V_3)$ is a $\mathbb{P}$-isometry if and only if $(V_1, V_2/2)$ is a $\mathbb{B}_2$-isometry and $(V_2, V_3)$ is a $\Gamma$-isometry.
\end{thm}

	\begin{cor}\label{cor_P_Q_iso}
		Let $A, X_3, S$ and $P$ be operators on a Hilbert space $\HS$. Then the following holds.
		\begin{enumerate}
			\item $(A, S, P)$ is a $\Pe$-isometry if and only if $(A, S\slash 2, S\slash 2, P)$ is an $\Q$-isometry.
			
			\item $(A, X_3)$ is a commuting pair of isometries if and only if $(A, 0, 0, X_3)$ is an $\Q$-isometry.
		\end{enumerate}
	\end{cor}
	
	Our aim now is to obtain a Wold type decomposition for an $\Q$-isometry. To do so, we recall from \cite{Bhattacharyya-01} a Wold type decomposition for an $\E$-isometry. Before this, we mention that a pure $\E$-isometry (in the sense of Definition \ref{defn:H-unitary}) is nothing but an $\E$-isometry with no $\E$-unitary part whereas in \cite{Bhattacharyya-01}, an $\E$-isometry is referred to as pure if its third component is a pure isometry. However, these two notions of pure $\E$-isometry coincide as discussed in the remark below.

	\begin{rem}\label{rem_pure_E} Let $(V_1, V_2, V_3)$ be a pure $\E$-isometry (in the sense of Definition \ref{defn:H-unitary}) on a Hilbert space $\HS$. We have by Theorem \ref{thm_E_isometry} that $V_3$ is an isometry. It follows from the Wold decomposition of an isometry (see Theorem 1.1 in \cite{Nagy}) that there are closed reducing subspaces $\HS_1, \HS_2$ for $V_3$  such that $\HS=\HS_1 \oplus \HS_2$, $V_3|_{\HS_1}$ is a unitary and $V_3|_{\HS_2}$ is a pure isometry. Following the proof of Theorem 5.6 in \cite{Bhattacharyya-01}, we have that $\HS_1, \HS_2$ are joint reducing subspaces for $V_1, V_2, V_3$. Consequently, it follows from Theorem \ref{thm_E_unitary} that $(V_1|_{\HS_1}, V_2|_{\HS_1}, V_3|_{\HS_1})$ is an $\E$-unitary. Since $(V_1, V_2, V_3)$ has no $\E$-unitary part, we must have $\HS_1=\{0\}$ and so, $\HS=\HS_2$. Therefore, $V_3$ is a pure isometry. Conversely, let $\underline{V}=(V_1, V_2, V_3)$ be an $\E$-isometry on a Hilbert space $\HS$ with $V_3$ being a pure isometry. Let $\mathcal{L} \subseteq \HS$ be a joint reducing subspace for $V_1, V_2, V_3$ such that $(V_1|_{\mathcal{L}}, V_2|_{\mathcal{L}}, V_3|_{\mathcal{L}})$ is an $\E$-unitary. By Theorem \ref{thm_E_unitary}, $V_3|_{\mathcal{L}}$ is a unitary and so, $\mathcal{L}=\{0\}$ since $V_3$ is a pure isometry. Hence, $\underline{V}$ is an $\Q$-isometry in the sense of Definition \ref{defn:H-unitary}. \qed
	\end{rem}

	The above remark is made to avoid any confusion between the two notions of pure $\E$-isometries (appearing here and in \cite{Bhattacharyya-01}), which turn out to be equivalent. Being equipped with this, we recall the following result for a Wold type decomposition of an $\E$-isometry.
	
	\begin{thm}[\cite{Bhattacharyya-01}, Theorem 5.6] \label{Wold_E_Isometry}
		Let $(V_1, V_2, V_3)$ be an $\E$-isometry on a Hilbert space $\HS$. Let $\HS=\HS_1 \oplus \HS_2$ be the Wold decomposition of $V_3$ such that $V_3|_{\HS_1}$ is a unitary and $V_3|_{\HS_2}$ is a pure isometry. Then $\mathcal{H}_1$ and $\mathcal{H} _2$ are reducing subspaces for $V_1, V_2, V_3$ and the following hold.
		\begin{enumerate}
			\item $(V_1|_{\HS_1}, V_2|_{\HS_1}, V_3|_{\HS_1})$ is an $\E$-unitary.
			\item $(V_1|_{\HS_2}, V_2|_{\HS_2}, V_3|_{\HS_2})$ is a pure $\E$-isometry.
		\end{enumerate}
	\end{thm}
	
We are now in a position to present a Wold type decomposition for $\Q$-isometries, which is another main result of this section.	
	
	\begin{thm}\textbf{(Wold decomposition of an $\mathbb{H}$-isometry).}\label{Wold}
		Let $(V_0, V_1, V_2, V_3)$ be an $\mathbb{H}$-isometry on a Hilbert space $\mathcal{H}$. Then there is a unique orthogonal decomposition $\mathcal{H}=\mathcal{H}^{(u)} \oplus \mathcal{H}^{(c)}$ such that $\mathcal{H}^{(u)}, \mathcal{H}^{(c)}$ are reducing subspaces of $V_0, V_1, V_2, V_3$ and the following hold.
		\begin{enumerate}
			\item $(V_0|_{\HS^{(u)}}, V_1|_{\HS^{(u)}}, V_2|_{\HS^{(u)}},V_3|_{\HS^{(u)}})$ is an $\Q$-unitary.
			\item   $(V_0|_{\HS^{(c)}}, V_1|_{\HS^{(c)}}, V_2|_{\HS^{(c)}},V_3|_{\HS^{(c)}})$ is a pure $\Q$-isometry.
		\end{enumerate}
		Also, there exists a further orthogonal decomposition $\HS^{(c)}= \HS_1^{(c)} \oplus \HS_2^{(c)}$ such that $\HS_1^{(c)}, \HS_2^{(c)}$ are reducing subspaces for $V_0, V_1, V_2, V_3$ and $V_3|_{\HS_2^{(c)}}$ is a pure isometry.		
	\end{thm}
	
	\begin{proof}
		We have by Theorem \ref{thm_Q_isometry} that $(V_1, V_2, V_3)$ is an $\E$-isometry and so, by Theorem \ref{thm_E_isometry}, $V_3$ is an isometry. Let $\HS_1 \oplus \HS_2$ be the Wold type decomposition of $V_3$ such that $V_3|_{\HS_1}$ is a unitary and $V_3|_{\HS_2}$ is a pure isometry. Indeed, the space $\HS_1$ is given by 
		\[
		\HS_1=\{x \in \HS : V_3^{*n}V_3^n x=V_3^nV_3^{*n}x=x  \ \ \text{for} \ \ n=0, 1, 2, \dotsc\}.
		\]
		By Theorem \ref{Wold_E_Isometry}, $\mathcal{H}_1, \mathcal{H}_2$ are joint reducing subspaces for $V_1, V_2, V_3$ such that $(V_1|_{\HS_1}, V_2|_{\HS_1}, V_3|_{\HS_1})$ is an $\E$-unitary and $(V_1|_{\HS_2}, V_2|_{\HS_2}, V_3|_{\HS_2})$ is a pure $\E$-isometry. Let 
		\[
		V_0=\begin{bmatrix}
			A_{0} & C\\
			C' & B_{0}\\
		\end{bmatrix}, \quad 	V_1=	\begin{bmatrix}
			A_1 & 0\\
			0 & B_1\\
		\end{bmatrix},		\quad V_2=	\begin{bmatrix}
			A_2 & 0\\
			0 & B_2\\
		\end{bmatrix} \quad \text{and} \quad V_3=\begin{bmatrix}
			A_3 & 0\\
			0 & B_3\\
		\end{bmatrix}.
		\]
		with respect to $\HS=\HS_1 \oplus \HS_2$. Since $V_0V_3=V_3V_0$, we have that $CB_3=A_3C$ and  $C'A_3=B_3C'$. Such an intertwining relation is possible only when $C=C'=0$ since $A_3$ is a unitary and $B_3^{*n} \to 0$ strongly as $n$ tends to infinity. Thus, $\HS_1, \HS_2$ are reducing subspaces for $V_0$ and so, $V_0=A_0 \oplus B_0$ with respect to $\mathcal{H}=\mathcal{H}_1 \oplus \mathcal{H}_2$. Evidently, $A_0$ is subnormal operator acting on $\HS_1$ and so, a further decomposition of $A_0$ into normal part and completely non-normal part is possible. Indeed, Lemma 3.1 in \cite{Morrel} gives that the space
		\[
		\mathcal{H}^{(u)}=\overset{\infty}{\underset{n=0}{\bigcap}}\text{Ker}\left(A_0^{*n}A_0^n-A_0^nA_0^{*n}\right)=\left\{x \in \HS_1 \ : \ A_0^{*n}A_0^nx=A_0^nA_0^{*n}x \ \ \text{for} \ \ n=0, 1, 2, \dotsc\right\}
		\]
		is the maximal reducing subspace of $A_0$ such that $A_0|_{\HS^{(u)}}$ is normal. Note that 
		\[
		\mathcal{H}^{(u)}=\bigg\{x \in \HS \ : \ V_3^{*n}V_3^n x=V_3^nV_3^{*n}x=x \ \ \& \ \ V_0^{*n}V_0^nx=V_0^nV_0^{*n}x \ \ \text{for} \ \ n=0, 1, 2, \dotsc \bigg\} \subseteq \HS_1.
		\]
		Evidently, $(A_1, A_2, A_3)$ is a commuting triple of normal operators acting on $\HS_1$ and  $A_jA_0=A_0A_j$ for $1 \leq j \leq 3$. We have by Fuglede's theorem \cite{Fuglede} that $A_j^*A_0=A_0A_j^*$ for $1 \leq j \leq 3$ and so, $(A_0, A_1, A_2, A_3)$ is a doubly commuting quadruple of operators on $\HS_1$. Then
		\[
		A_j\left(A_0^{*n}A_0^n-A_0^nA_0^{*n}\right)=\left(A_0^{*n}A_0^n-A_0^nA_0^{*n}\right)A_j \quad \text{and} \quad A_j^*\left(A_0^{*n}A_0^n-A_0^nA_0^{*n}\right)=\left(A_0^{*n}A_0^n-A_0^nA_0^{*n}\right)A_j^*
		\]
		for $1 \leq j \leq 3$. Therefore, $\mathcal{H}^{(u)}$ is a joint reducing subspace of $A_0, A_1, A_2$ and $A_3$. Let us define 
		\[
		\underline{U}=(U_0, U_1, U_2, U_3 )= (A_0|_{\HS^{(u)}}, A_1|_{\HS^{(u)}}, A_2|_{\HS^{(u)}}, A_3|_{\HS^{(u)}}).
		\]
		Then $\underline{U}$ is a commuting quadruple of normal operators and $(A_1|_{\HS^{(u)}}, A_2|_{\HS^{(u)}}, A_3|_{\HS^{(u)}})$ is an $\E$-unitary. For $x \in \HS^{(u)}$, we have that $U_0^*U_0x+U_1^*U_1x=A_0^*A_0x+A_1^*A_1x=x$. It now follows from Theorem \ref{thm_Q_unitary} that $\underline{U}$ is an $\Q$-unitary. Let $\HS' \subseteq \HS$ be a joint reducing subspace for $V_0, V_1, V_2, V_3$ such that 
		$\underline{U}'=(V_0|_{\HS'}, V_1|_{\HS'}, V_2|_{\HS'}, V_3|_{\HS'})$ is an $\Q$-unitary. Let $U_j'=V_j|_{\HS'}$ for $0 \leq j \leq 3$. By Theorem \ref{thm_Q_unitary}, $(U_1', U_2', U_3')$ is an $\E$-unitary and by Theorem \ref{thm_E_unitary}, $U_3'$ is a unitary. Since $\HS_1$ is the maximal closed subspace of $\HS$ that reduces $V_3$ to a unitary, we have that $\HS' \subseteq \HS_1$. Also, $A_0|_{\HS'}=V_0|_{\HS'}=U_0'$ is a normal operator. Since $\HS^{(u)}$ is the maximal closed subspace of $\HS_1$ reducing $A_0$ to a normal operator, we have that $\HS' \subseteq \HS^{(u)}$. Hence, $\HS^{(u)}$ is the maximal joint reducing subspace for $V_0, V_1, V_2, V_3$ restricted to which $(V_0, V_1, V_2, V_3)$ acts as an $\Q$-unitary. Let $\mathcal{H}^{(c)}=\HS \ominus \HS^{(u)}$. Then $\mathcal{H}^{(c)}=(\HS_1\ominus \HS_1^{(u)}) \oplus \HS_2$ and $V_3|_{\HS_2}$ is a pure isometry. The maximality of $\HS^{(u)}$ implies that $(V_0|_{\HS^{(c)}}, V_1|_{\HS^{(c)}}, V_2|_{\HS^{(c)}}, V_3|_{\HS^{(c)}})$ is a pure $\Q$-isometry. The uniqueness part is also an immediate consequence of the maximality of $\HS^{(u)}$. The proof is now complete. 
	\end{proof}
	
	We have by Remark \ref{rem_pure_E} that an $\E$-isometry is pure if and only if its last component is a pure isometry. One can ask if an analogous statement holds for an $\Q$-isometry. It is not difficult to see that if $(V_0, V_1, V_2, V_3)$ is an $\Q$-isometry with $V_3$ being a pure isometry, then $(V_0, V_1, V_2, V_3)$ is a pure $\Q$-isometry. Indeed, by Theorem \ref{Wold}, there is an orthogonal decomposition $\mathcal{H}=\mathcal{H}^{(u)} \oplus \mathcal{H}^{(c)}$ such that $\mathcal{H}^{(u)}, \mathcal{H}^{(c)}$ are joint reducing subspaces for $V_0, V_1, V_2, V_3$ and the following hold.
	\begin{enumerate}
		\item $(V_0|_{\HS^{(u)}}, V_1|_{\HS^{(u)}}, V_2|_{\HS^{(u)}},V_3|_{\HS^{(u)}})$ is an $\Q$-unitary.
		\item   $(V_0|_{\HS^{(c)}}, V_1|_{\HS^{(c)}}, V_2|_{\HS^{(c)}},V_3|_{\HS^{(c)}})$ is a pure $\Q$-isometry.
	\end{enumerate}		
	It follows from Theorem \ref{thm_Q_unitary} that $V_3|_{\HS^{(u)}}$ is a unitary. Since $V_3$ is pure, we have that $\HS^{(u)}=\{0\}$ and so, we have arrived at the following result. 
	
	\begin{prop}\label{prop_pureQ}
		If $\underline{V}=(V_0, V_1, V_2, V_3)$ is an $\Q$-isometry and $V_3$ is a pure isometry, then $\underline{V}$ is pure.
	\end{prop}
	
	However, the converse to the above result is not true. We refer to Example \ref{exm:imp} here. 
	\begin{eg} 
		Let $T_z$ be the unilateral shift on $\ell^2(\mathbb{N})$. We have by part (2) of Corollary \ref{cor_P_Q_iso} that $\underline{V}=(T_z, 0, 0, I)$ is an $\Q$-isometry. Clearly, the last component of $\underline{V}$ is not a pure isometry. Let $\mathcal{L} \subseteq \ell^2(\N)$ be a joint reducing subspace of $\underline{V}$ such that $\underline{V}$ restricted to $\mathcal{L}$ is an $\Q$-unitary. In particular, $\mathcal{L}$ is a reducing subspace of $T_z$ and $T_z|_{\mathcal{L}}$ is a normal operator. It is well-known that the only reducing subspaces for $T_z$ are $\ell^2(\N)$ and $\{0\}$. Since $T_z|_{\mathcal{L}}$ is normal and $\mathcal{L} \ne \ell^2(\mathbb{N})$, we have that $\mathcal{L}=\{0\}$. Hence, $\underline{V}$ is a pure $\Q$-isometry but its last component is not a pure isometry. \qed 
	\end{eg} 
The above example also shows that if $(V_0, V_1, V_2, V_3)$ is a pure $\Q$-isometry, then $(V_1, V_2, V_3)$ need not be a pure $\E$-isometry. These interesting observations motivate us to study the class of $\Q$-isometries with the last component being a pure isometry. We refer to Section 6 in \cite{JindalII} for a similar discussion on the pentablock isometries with the last component being a pure isometry. In this direction, we recall from \cite{Bhattacharyya-01} the notion of fundamental operators of an $\E$-contraction, and an explicit model for pure $\E$-isometries which was obtained in \cite{PalI}. For an $\E$-contraction $(X_1, X_2, X_3)$, the fundamental equations given by 
	\[
	X_1-X_2^*X_3=D_{X_3}F_1D_{X_3} \qquad \text{and} \qquad X_2-X_1^*X_3=D_{X_3}F_2D_{X_3}
	\]
	have unique solutions $F_1$ and $F_2$ in $\mathcal{B}(\mathcal{D}_{X_3})$. Also, the operator $F_1+F_2z$ has numerical radius at most $1$ for all $z \in \DC$. Moreover,  $F_1$ and $F_2$ satisfy the pair of operator equations
	\[
	D_{X_3}X_1=F_1D_{X_3}+F_2^*D_{X_3}X_3 \quad \text{and} \quad D_{X_3}X_2=F_2D_{X_3}+F_1^*D_{X_3}X_3.
	\]
\begin{thm}[\cite{PalI}, Theorem 3.3]\label{thm_pure_E_iso}
		Let $(V_1, V_2, V_3)$ be a commuting triple of operators on a Hilbert space $\HS$. If $(V_1, V_2, V_3)$ is a pure $\E$-isometry, then there exists a unitary operator $U: \HS \to H^2(\mathcal{D}_{V_3^*})$ such that 
		\[
		V_1=U^*T_\varphi U, \quad 	V_2=U^*T_\psi U \quad \text{and} \quad V_3=U^*T_zU,
		\]
		where $\varphi(z)=F_1^*+F_2z, \psi(z)=F_2^*+F_1z, \ z \in \D$ and $F_1, F_2$ are the fundamental operators of $(V_1^*, V_2^*, V_3^*)$ such that 
		\begin{enumerate}
			\item $[F_1, F_2]=0$ and $[F_1^*, F_1]=[F_2^*, F_2]$,
			\item $\|F_1^*+F_2z\|_{\infty, \DC} \leq 1$.
		\end{enumerate}
		Conversely, if $F_1$ and $F_2$ are operators on a Hilbert space $\mathcal{L}$ satisfying the above two conditions, then $(T_{F_1^*+F_2z}, T_{F_2^*+F_1z}, T_z)$ on $H^2(\mathcal{L})$ is a pure $\E$-isometry.	
	\end{thm}
Let $\underline{V}=(V_0, V_1, V_2, V_3)$ be a pure $\Q$-isometry and let $V_3$ be a pure isometry. We have by Theorem \ref{thm_Q_isometry} and Remark \ref{rem_pure_E} that $(V_1, V_2, V_3)$ is a pure $\E$-isometry. It follows from Theorem \ref{thm_pure_E_iso} that $(V_1, V_2, V_3)$ is unitarily equivalent to the commuting triple $(T_{F_1^*+F_2z}, T_{F_2^*+F_1z}, T_z)$ on the vector-valued Hardy space $H^2(\mathcal{D}_{V_3^*})$, where $F_1, F_2$ are the fundamental operators of $(V_1^*, V_2^*, V_3^*)$. Since $V_0$ commutes with $V_3$, there exists a bounded holomorphic map $\phi : \D \to \mathcal{B}(\mathcal{D}_{V_3^*})$ such that $V_0$ and $T_\phi$ are unitarily equivalent on $H^2(\mathcal{D}_{V_3^*})$. Therefore, $\underline{V}$ is unitarily equivalent to the quadruple $(T_\phi, T_{F_1^*+F_2z}, T_{F_2^*+F_1z}, T_z)$ on $H^2(\mathcal{D}_{V_3^*})$. In this direction, we want to have a characterization for such a quadruple to become an $\Q$-isometry when $\phi(z)=G_0+zG_1$ for some operators $G_0, G_1$ on $\mathcal{D}_{V_3^*}$. The motivation of this comes from the operator Fej\'{e}r-Riesz theorem \cite{Dritschel} which we briefly explain here. Let $F_1, F_2$ be two operators on a Hilbert space $\mathcal{L}$ such that we have the following:
	\begin{enumerate}
		\item $[F_1, F_2]=0$ and $[F_1, F_1^*]=[F_2, F_2^*]$,
		\item $\|F_1^*+F_2z\|_{\infty, \DC}\leq 1$. 		
	\end{enumerate}		
	It follows from  Theorem \ref{thm_pure_E_iso} that $(T_{F_1^*+F_2z}, T_{F_2^*+F_1z}, T_z)$ is a pure $\E$-isometry. Since $T_{F_1^*+F_2z}$ is a contraction, we have that 
	$
	I-T_{F_1^*+F_2z}^*T_{F_1^*+F_2z} \geq 0$ and so, $I-(F_1^*+F_2z)^*(F_1^*+F_2z) \geq 0$ for all $z  \in \T$. Now, operator Fej\'{e}r-Riesz Theorem \cite{Dritschel} ensures the existence of $G_0$ and $G_1$ in $\mathcal{B}(\mathcal{L})$ such that  
	\[
	(G_0+zG_1)^*(G_0+zG_1)=I-(F_1^*+F_2z)^*(F_1^*+F_2z)  \quad \text{for all $z  \in \T$}.
	\]
	Consequently, $\underline{V}=(T_{G_0+G_1z}, T_{F_1^*+F_2z}, T_{F_2^*+F_1z}, T_z)$ on $H^2(\mathcal{L})$ becomes a pure $\Q$-isometry if 
	\[
	[F_1^*+F_2z, \ G_0+G_1z]=0 \quad \text{and} \quad [F_2^*+F_1z, \ G_0+G_1z]=0 \quad \text{for all $z \in \T$}.
	\]
	The above condition guarantees the commutativity of the operators in $\underline{V}$. Our next result provides necessary and sufficient conditions for such a quadruple to become a pure $\Q$-isometry.
	
	\begin{thm}\label{thm_pureQ_isom}
		Let $G_0, G_1, F_1, F_2$ be operators on a Hilbert space $\mathcal{L}$. Then the quadruple  
		\[
		\underline{V}=(V_0 ,V_1, V_2, V_3)=(T_{G_0+G_1z}, T_{F_1^*+F_2z}, T_{F_2^*+F_1z}, T_z) \quad \text{on} \ H^2(\mathcal{L})
		\] 
		is a pure $\Q$-isometry if and only if $\|F_1^*+F_2z\|_{\infty, \DC} \leq 1$ and the following hold:
		
		\begin{minipage}[t]{0.25\textwidth}
			\smallskip
			\begin{equation*}
				\begin{split}
					& 1. \ [F_1, F_2]=0,\\
					& 4. \ [F_1, G_0]=[G_1, F_2^*] , \\
					& 7.  \ [F_2, G_0]=[G_1, F_1^*] , \\
					& 10.  \ G_0^*G_1+F_1F_2=0. \\
				\end{split}
			\end{equation*}	
		\end{minipage}
		\begin{minipage}[t]{0.3\textwidth}
			\smallskip
			\begin{equation*}
				\begin{split}
					& 2. \ [F_1^*, F_1]=[F_2^*, F_2],\\
					& 5.  \ [F_1, G_1]=0 , \\
					& 8. \ [F_2, G_1]=0,\\ 
				\end{split}
			\end{equation*}	
		\end{minipage}
		\begin{minipage}[t]{0.3\textwidth}
			\smallskip
			\begin{equation*}
				\begin{split}
					& 3. \ [F_2^*, G_0]=0,\\
					& 6.  \ [F_1^*, G_0]=0, \\
					& 9.  \ G_0^*G_0+G_1^*G_1=I-F_1F_1^*-F_2^*F_2 , \\
				\end{split}
			\end{equation*}	
		\end{minipage}
	\end{thm}
	
	\begin{proof}
		We shall frequently use the natural isomorphism between the Hardy space $H^2(\mathcal{L})$ of $\mathcal{L}$-valued functions on the unit disc $\D$ and $\ell^2(\mathcal{L})$. We use this identification without any further mention. We can write the operators as  
		\begin{small} 	
			\begin{equation*}
				V_0=  \begin{bmatrix} 
					G_0 & 0 & 0 &  \dotsc \\
					G_1 & G_0 & 0 & \dotsc \\
					0 & G_1 & G_0 &  \dotsc\\
					\dotsc & \dotsc & \dotsc  & \dotsc \\
				\end{bmatrix}, \ 
				V_1=  \begin{bmatrix} 
					F_1^* & 0 & 0 &  \dotsc \\
					F_2 & F_1^* & 0 & \dotsc \\
					0 & F_2 & F_1^* &  \dotsc\\
					\dotsc & \dotsc & \dotsc  & \dotsc \\
				\end{bmatrix},  \ 
				V_2=  \begin{bmatrix} 
					F_2^* & 0 & 0 &  \dotsc \\
					F_1 & F_2^* & 0 & \dotsc \\
					0 & F_1 & F_2^* &  \dotsc\\
					\dotsc & \dotsc & \dotsc  & \dotsc \\
				\end{bmatrix},  \ 
				V_3=  \begin{bmatrix} 
					0 & 0 & 0 &  \dotsc \\
					I & 0 &  0 & \dotsc \\
					0 & I & 0 &  \dotsc\\
					\dotsc & \dotsc  & \dotsc & \dotsc \\
				\end{bmatrix}
			\end{equation*}
		\end{small}
		with respect to the decomposition	$\ell^2(\mathcal{L})=\mathcal{L} \oplus \mathcal{L} \oplus \dotsc$. It is easy to see that $V_3$ commutes with $V_0, V_1$ and $V_2$. A few steps of simple calculations give that the operators of the form
		\begin{equation*}
			\begin{bmatrix} 
				P & 0 & 0 &  \dotsc \\
				Q & P & 0 & \dotsc \\
				0 & Q & P &  \dotsc\\
				\dotsc & \dotsc & \dotsc  & \dotsc \\
			\end{bmatrix} \quad \text{and} \quad  
			\begin{bmatrix} 
				R & 0 & 0 &  \dotsc \\
				S & R & 0 & \dotsc \\
				0 & S & R &  \dotsc\\
				\dotsc & \dotsc & \dotsc  & \dotsc \\
			\end{bmatrix} \ 
		\end{equation*} 
		commute if and only if $[P, R]=[Q, S]=0,$ and $[Q, R]=[S, P]$. Consequently, $V_0V_1=V_1V_0$ if and only if conditions $(6)-(8)$ in the statement of the theorem hold. Similarly, $
		V_0V_2=V_2V_0$ if and only if conditions $(3)-(5)$ of the statement hold. The last commutativity condition is $V_1V_2=V_2V_1$ which holds if and only if conditions $(1)$ and $(2)$ of the statement hold. Again, some simple computations give that
		\[
		V_0^*V_0=\begin{bmatrix}
			G_0^*G_0+G_1^*G_1 & G_1^*G_0 & 0 & \dotsc \\
			G_0^*G_1 & G_0^*G_0+G_1^*G_1 & G_1^*G_0 & \dotsc \\ 
			0 & G_0^*G_1 & G_0^*G_0+G_1^*G_1 & \dotsc \\
			\dotsc & \dotsc & \dotsc & \dotsc \\
		\end{bmatrix}
		\]
		and  
		\[
		V_1^*V_1=\begin{bmatrix}
			F_1F_1^*+F_2^*F_2 & F_2^*F_1^* & 0 & \dotsc \\
			F_1F_2 & F_1F_1^*+F_2^*F_2 & F_2^*F_1^* & \dotsc \\ 
			0 & F_1F_2 & F_1F_1^*+F_2^*F_2 & \dotsc \\
			\dotsc & \dotsc & \dotsc & \dotsc \\
		\end{bmatrix}.
		\] 
		Hence, $V_0^*V_0+V_1^*V_1=I$ if and only if conditions $(9)$ and $(10)$ in the statement of the theorem holds. We use these equivalent conditions to obtain the desired conclusion.
		
		\smallskip

		Assume that $\underline{V}$ is a pure $\Q$-isometry. The commutativity hypothesis of the operators in $\underline{V}$ gives conditions $(1)-(8)$ of the statement.  By Theorem \ref{thm_Q_isometry}, $V_0^*V_0+V_1^*V_1=I$ which gives conditions $(9)-(10)$. It follows from Theorem \ref{thm_Q_isometry} and Remark \ref{rem_pure_E} that $(V_1, V_2, V_3)$ is a pure $\E$-isometry.  We have by Theorem \ref{thm_pure_E_iso} that  $\|F_1^*+F_2z\|_{\infty, \DC} \leq 1$. Conversely, conditions $(1), (2)$ and $\|F_1^*+F_2z\|_{\infty, \DC} \leq 1$ implies that $(V_1, V_2, V_3)$ is a pure $\E$-isometry by virtue of Theorem \ref{thm_pure_E_iso}. Also, conditions $(1)-(8)$ give that $\underline{V}$ is a commuting quadruple and conditions $(9)-(10)$ yield that $V_0^*V_0+V_1^*V_1=I$. By Theorem \ref{thm_Q_isometry} and Proposition \ref{prop_pureQ}, $\underline{V}$ is a pure $\Q$-isometry. The proof is now complete.
	\end{proof}

	In general, we do not know if the system of operator equations as in the statement of the above theorem admit a solution. However, we can find a solution $(G_0, G_1)$ when $F_1, F_2$ are commuting normal operators such that $\|F_1^*+F_2z\|_{\infty, \DC} \leq 1$.
	
	\begin{thm}\label{thm_fo_normalI}
		Let $F_1, F_2$ be commuting normal operators acting on a Hilbert space $\mathcal{L}$ such that $\|F_1^*+F_2z\|_{\infty, \DC} \leq 1$. Then there exists $G_0, G_1 \in \mathcal{B}(\mathcal{L})$ such that $(T_{G_0+G_1z}, T_{F_1^*+F_2z}, T_{F_2^*+F_1z}, T_z)$ on $H^2(\mathcal{L})$ is a pure $\Q$-isometry.
	\end{thm} 
	
	\begin{proof}
		It suffices to find $G_0, G_1 \in \mathcal{B}(\mathcal{L})$ satisfying all the equations in the statement of Theorem \ref{thm_pureQ_isom}. By spectral theorem, there exists a unique spectral measure $E$ on $\sigma_T(F_1, F_2)$ such that 
		\[
		F_1=\int_{\sigma_T(F_1, F_2)} \textbf{z}_1 dE \quad \text{and} \quad 	F_2=\int_{\sigma_T(F_1, F_2)} \textbf{z}_2dE, 
		\]
		where $\textbf{z}_1, \textbf{z}_2$ are the natural coordinate maps on $\C^2$. Let $(\alpha_1, \alpha_2) \in \sigma_T(F_1, F_2)$. We show that $|\alpha_1|+|\alpha_2| \leq 1$. Let $z \in \DC$.  We have by spectral mapping principle that $\overline{\alpha}_1+z\alpha_2 \in \sigma(F_1^*+F_2z)$. Since $\|F_1^*+F_2z\| \leq 1$, we have that $|\overline{\alpha}_1+z\alpha_2| \leq 1$ for all $z \in \DC$. Thus, $|\alpha_1| +|\alpha_2| \leq 1$ if at least one of $\alpha_1$ or $\alpha_2$ is zero. Assume that both $\alpha_1$ and $\alpha_2$ are non-zero. Define $\displaystyle z=\frac{\overline{\alpha}_2|\alpha_1|}{|\alpha_2|\alpha_1}$. Then
		$
|\alpha_1|+|\alpha_2|=|\overline{\alpha}_1+z\alpha_2| \leq 1
		$
		and so, $\sigma_T(F_1, F_2) \subseteq \{(\alpha_1, \alpha_2) \in \C^2 : |\alpha_1|+|\alpha_2| \leq 1\}$. Consequently, the real-valued maps $\alpha, \beta$ on $\sigma_T(F_1, F_2)$ given by
		$
		\displaystyle \alpha(z_1, z_2)=\sqrt{1-(|z_1|-|z_2|)^2}$ and $\displaystyle \beta(z_1, z_2)=\sqrt{1-(|z_1|+|z_2|)^2}$ are continuous. Consider the functions on $\sigma_T(F_1, F_2)$ defined as
		\[
		f(z_1, z_2)=\frac{1}{2}\left[\alpha(z_1, z_2)+\beta(z_1, z_2)\right] \quad \text{and} \quad 	g(z_1, z_2)=\left\{
		\begin{array}{ll}
			\displaystyle  \frac{-z_1z_2}{2|z_1z_2|}\left[\alpha(z_1, z_2)-\beta(z_1, z_2)\right], & z_1z_2 \neq 0 \\ \medskip
			0, & z_1z_2=0 \,.\\
		\end{array} 
		\right. 
		\]
		Clearly, $f$ is continuous on $\sigma_T(F_1, F_2)$. Also, $g$ is Borel measurable  on $\sigma_T(F_1, F_2)$ since the map
		\[
		h: \sigma_T(F_1, F_2) \to \C, \quad h(z_1, z_2)=\left\{
		\begin{array}{ll}
			\displaystyle  \frac{-z_1z_2}{|z_1z_2|}, & z_1z_2 \neq 0 \\ \medskip
			0, & z_1z_2=0 \,.\\
		\end{array} 
		\right. 
		\] 	
		is Borel measurable. A few tedious but routine computations give that
		\begin{equation}\label{eqn_norm}
			|f(z_1, z_2)|^2+|g(z_1, z_2)|^2=1-|z_1|^2-|z_2|^2 \quad \text{and} \quad \overline{f(z_1, z_2)}g(z_1, z_2)+z_1z_2=0
		\end{equation}
		for all $(z_1, z_2) \in \sigma_T(F_1, F_2)$. Let $G_0=f(F_1, F_2)$ and $G_1=g(F_1, F_2)$. Consequently,  $G_0, G_1, F_1, F_2$ are commuting normal operators and so, conditions $(1)-(8)$ in Theorem \ref{thm_pureQ_isom} hold. An application of the spectral theorem for commuting normal operators and \eqref{eqn_norm} gives that 
		$
		G_0^*G_0+G_1^*G_1=I-F_1^*F_1-F_2^*F_2$ and $ G_0^*G_1+F_1F_2=0$, which are conditions $(9)-(10)$ of Theorem \ref{thm_pureQ_isom}. 
	\end{proof}
	
	We present an example of an $\Q$-isometry with its last component being a pure isometry. 
	
	\begin{eg}
		Let $A, B$ be commuting isometries and consider the commuting quadruple 
		\[
		\underline{V}=(V_0, V_1, V_2, V_3)= \left(\frac{1}{2}(A-B), \frac{1}{2}(A+B), \frac{1}{2}(A+B), AB\right).
		\]
		Clearly, $V_1=V_2^*V_3$, $V_3$ is an isometry and $\|V_2\| \leq 1$. By Theorem \ref{thm_E_isometry}, $(V_1, V_2, V_3)$ is an $\E$-isometry. Evidently, 
				$V_0^*V_0+V_1^*V_1=I$. It now follows from Theorem \ref{thm_Q_isometry} that $\underline{V}$ is an $\Q$-isometry. Thus, $AB$ is a pure isometry if either $A$ or $B$ is a pure isometry. 	\qed 
	\end{eg}
	
It follows from Theorem \ref{thm_pure_E_iso} that a pure $\E$-isometry $(V_1, V_2, V_3)$ can be modeled as a commuting triple $(T_{F_1^*+F_2z}, T_{F_2^*+F_1z}, T_z)$ on $H^2(\mathcal{D}_{V_3^*})$, where $F_1, F_2$ are the fundamental operators of $(V_1^*, V_2^*, V_3^*)$. Further, if $F_1, F_2$ are commuting normal operators with $\|F_1^*+F_2z\|_{\infty, \DC} \leq 1$, then Theorem \ref{thm_fo_normalI} ensures the existence of operators $G_0, G_1$ on $\mathcal{D}_{V_3^*}$ such that $(T_{G_0+G_1z}, T_{F_1^*+F_2z}, T_{F_2^*+F_1z}, T_z)$ is a pure $\Q$-isometry. We conclude this section by showing that these conditions on the fundamental operators of an $\E$-contraction $(X_1, X_2, X_3)$ always hold when $X_1, X_2, X_3$ are normal operators. For this purpose, let us recall from \cite{Nagy} the notion of the characteristic function of a contraction $T$. For a contraction $T$ on a Hilbert space $\HS$, let $\Lambda_T$ be the collection of all complex numbers such that $I-zT^*$ is invertible. For $z \in \Lambda_T$, the characteristic function of $T$ is defined as 
\[
\Theta_{T}(z)=[-T+z D_{T^*}(I-zT^*)^{-1}D_{T}]|_{\mathcal{D}_{T}}.
\]
To the best of our knowledge, the following result has not appeared in the literature.

	\begin{prop}\label{prop_normal_E_fo}
		Let $(X_1, X_2, X_3)$ be an $\E$-contraction consisting of normal operators. Then its fundamental operators $F_1, F_2$ are commuting normal operators and $\|F_1^*+F_2z\|_{\infty, \DC} \leq 1$.
	\end{prop}
	
	\begin{proof}
		Since the last component of an $\E$-contraction is a contraction, we have that $\|X_3\| \leq 1$. Clearly, $D_{X_3}=D_{X_3^*}$ and so, $\mathcal{D}_{X_3}=\mathcal{D}_{X_3^*}$. Consider the characteristic function of $X_3$ given by
		\[
		\Theta_{X_3}(z)=[-X_3+zD_{X_3^*}(I-zX_3^*)^{-1}D_{X_3}]|_{\mathcal{D}_{X_3}} \quad \text{	for $z \in \D$}.
		\]
		Let $G_1, G_2$ be the fundamental operators of $(X_1^*, X_2^*, X_3^*)$. It follows from Theorem 3 in \cite{Bhatt_Lata_Sau} that 
		\begin{align}\label{eqn_Lata}
			(G_1^*+G_2z)\Theta_{X_3}(z)=\Theta_{X_3}(z)(F_1+F_2^*z) \quad \text{and} \quad (G_2^*+G_1z)\Theta_{X_3}(z)=\Theta_{X_3}(z)(F_2+F_1^*z)
		\end{align}  
		for all $z \in \D$. Since $X_1, X_2, X_3$ are commuting normal operators, we have that
		\[
		D_{X_3}G_1D_{X_3}=(X_1-X_2^*X_3)^*=D_{X_3}F_1^*D_{X_3} \quad \text{and} \quad 	D_{X_3}G_2D_{X_3}=(X_2-X_1^*X_3)^*=D_{X_3}F_2^*D_{X_3}.
		\]
		By uniqueness of fundamental operators, $(G_1, G_2)=(F_1^*, F_2^*)$. Substituting $z=0$ in \eqref{eqn_Lata}, we have
		\begin{align}\label{eqn_504}
			F_1X_3=X_3F_1 \quad \text{and} \quad F_2X_3=X_3F_2 \quad \text{on $\mathcal{D}_{X_3}$}.
		\end{align}
		Since $X_1, X_2, X_3$ are commuting normal operators, $D_{X_3}X_j=X_jD_{X_3}$ for $1 \leq j \leq 3$ and so, $\mathcal{D}_{X_3}$ is a joint reducing subspace for $X_1, X_2, X_3$. By \eqref{eqn_504} and Fuglede's theorem \cite{Fuglede}, we have that 
		\begin{align}\label{eqn_505}
			F_1^*X_3=X_3F_1^* \quad \text{and} \quad F_2^*X_3=X_3F_2^* \quad \text{on $\mathcal{D}_{X_3}$}.
		\end{align}
		It follows from \eqref{eqn_504} and \eqref{eqn_505} that
		$D_{X_3}^2$ doubly commutes with $F_1$ and $F_2$. By Lemma \ref{lem_basic}, we have
		\begin{align}\label{eqn_506}
			D_{X_3}F_j=F_jD_{X_3} \quad \text{and} \quad D_{X_3}F_j^*=F_j^*D_{X_3}  \quad \text{on $\mathcal{D}_{X_3}$}
		\end{align}
		for$j=1, 2$. For the normal operators $N=X_1-X_2^*X_3$ and $M=X_2-X_1^*X_3$, we have that
		\begin{align*}
			\begin{split}
				N^*N=(X_1^*-X_2X_3^*)(X_1-X_2^*X_3)
				=D_{X_3}F_1^*D_{X_3}^2F_1D_{X_3}
				=F_1^*F_1D_{X_3}^4 \quad [\text{by} \ \eqref{eqn_506}],
			\end{split}
		\end{align*}
		and 
		\begin{align*}
			\begin{split}
				M^*M=(X_2^*-X_1X_3^*)(X_2-X_1^*X_3)
				=D_{X_3}F_2^*D_{X_3}^2F_2D_{X_3}
				=F_2^*F_2D_{X_3}^4 \quad [\text{by} \ \eqref{eqn_506}].
			\end{split}
		\end{align*}
		Similarly, one can prove that $NN^*=F_1F_1^*D_{X_3}^4$ and $MM^*=F_2F_2^*D_{X_3}^4$. Thus, $F_i^*F_iD_{X_3}^4=F_iF_i^*D_{X_3}^4$ and by Lemma \ref{lem_basic}, $F_i^*F_iD_{X_3}=F_iF_i^*D_{X_3}$ for $i=1,2$. Since $F_1, F_2$ are operators on $\mathcal{D}_{X_3}$, it follows that $F_1$ and $F_2$ are normal operators. It is clear that $NM=MN$ since $X_1, X_2, X_3$ are commuting normal operators. Also, we have
		\begin{align*} 
			\begin{split}
				NM=(X_1-X_2^*X_3)(X_2-X_1^*X_3)
				=D_{X_3}F_1D_{X_3}^2F_2D_{X_3}
				=F_1F_2D_{X_3}^4 \quad [\text{by} \ \eqref{eqn_506}]
			\end{split}
		\end{align*}
		and 
		\begin{align*} 
			\begin{split}
				MN=(X_2-X_1^*X_3)(X_1-X_2^*X_3)
				=D_{X_3}F_2D_{X_3}^2F_1D_{X_3}
				=F_2F_1D_{X_3}^4 \quad [\text{by} \ \eqref{eqn_506}]
			\end{split}
		\end{align*}
		which gives that $F_2F_1D_{X_3}^4=F_1F_2D_{X_3}^4$. Again by Lemma \ref{lem_basic},  $F_2F_1D_{X_3}=F_1F_2D_{X_3}$ and so, $F_1F_2=F_2F_1$. It is remaining to show that $\|F_1^*+F_2z\|_{\infty, \DC} \leq 1$. Since $F_1, F_2$ are fundamental operators of $(X_1, X_2, X_3)$, the operator $F_1+F_2z$ has numerical radius at most $1$ for all $z \in \DC$. It follows from Lemma 2.6 in \cite{Pal-IV} that $\omega(F_1^*+F_2z) \leq 1$ for all $z \in \T$. Since $F_1^*+F_2z$ is a normal operator for any $z \in \DC$, its norm is same as its numerical radius and so, $\|F_1^*+F_2z\| \leq 1$ for all $z \in \T$. The desired conclusion now follows from the maximum modulus principle.
	\end{proof}
	
\section{An explicit $\Q$-isometric dilation and a concrete functional model for a class of $\Q$-contractions}\label{sec_dilation}

\vspace{0.2cm}

\noindent The success or failure of rational dilation on a domain is always an interesting yet highly challenging problem. There are various domains in the literature that have been studied in this context. For example, rational dilation succeeds on the bidisc $\D^2$ (see \cite{Ando, Nagy}), and on the symmetrized bidisc $\G_2$ (see \cite{AglerII, Bhattacharyya}). It is still unknown if the rational dilation succeeds on the tetrablock (see \cite{BallSau, PalI}) or on the pentablock (see \cite{PalN}). Nevertheless, conditional dilation results for the tetrablock were obtained in \cite{Bhattacharyya-01, Pal-IV}, and for the pentablock in \cite{PalN}. In this section, we provide a conditional dilation theorem for a subclass of $\Q$-contractions $(A, X_1, X_2, X_3)$ acting on a Hilbert space $\HS$ such that the last component $X_3$ is a pure contraction. Let us recall that a pure contraction $T$ is a contraction such that $X_3^{*n} \to 0$ strongly as $n \to \infty$. It was shown in \cite{Nagy} that every pure contraction $T$ defined on a Hilbert space $\HS$ is unitarily equivalent to the operator $P_{\mathcal{K}_T}(T_z \otimes I)|_{\mathcal{K}_T}$ on the Hilbert space 
\[
\mathcal{K}_T=(H^2(\D)\otimes \mathcal{D}_{T^*})\ominus M_{\Theta_T}(H^2(\D)\otimes \mathcal{D}_T),
\]  
where $T_z$ is the multiplication operator on $H^2(\D)$ and $M_{\Theta_T}$ is the multiplication operator from $H^2(\D) \otimes \mathcal{D}_T$ into $H^2(\D) \otimes \mathcal{D}_{T^*}$ corresponding to the multiplier $\Theta_T$, the characteristic function of $T$. Motivated by this, the first named author of this article established the following dilation theorem for a subclass of $\E$-contractions with the last component being a pure contraction.

\begin{thm}[\cite{Pal-IV}, Theorem 3.2]\label{thm_pure_E_dilation}
	Let $(X_1, X_2, X_3)$ be an $\E$-contraction with $X_3$ being a pure contraction on a Hilbert space $\HS$. Let $G_1, G_2$ be the fundamental operators of $(X_1^*, X_2^*, X_3^*)$ satisfying $[G_1, G_2]=0$ and $[G_1^*, G_1]=[G_2^*, G_2]$. Then the operator triple 
	\begin{equation}\label{eqn_709}
		\left(I\otimes G_1^*+T_z \otimes G_2, \ I\otimes G_2^*+T_z \otimes G_1, \ T_z \otimes I \right)
	\end{equation}
	on $H^2(\D) \otimes \mathcal{D}_{X_3^*}$ is a minimal pure $\E$-isometric dilation of $(X_1, X_2, X_3)$. 
\end{thm} 
As an application of Theorem \ref{thm_pure_E_dilation}, we produce a minimal $\Q$-isometric dilation for a subclass of $\Q$-contractions with their last components being a pure contraction. As a consequence of this dilation, we obtain a functional model for such $\Q$-contractions $(A, X_1, X_2, X_3)$ in terms of commuting Toeplitz operators on the vectorial Hardy space $H^2(\mathcal{D}_{X_3^*})$. We begin with a few preparatory results associated with $\Q$-contractions.

\begin{prop}
	An $\Q$-contraction $(A, X_1, X_2, X_3)$ admits an $\Q$-isometric dilation if and only if it has a minimal $\Q$-isometric dilation.
\end{prop}

\begin{proof}
	The converse is trivial. We prove that the forward part. Suppose $(A, X_1, X_2, X_3)$ is an $\Q$-contraction on a Hilbert space $\HS$. Let $(V, V_1, V_2, V_3)$ acting on a Hilbert space $\mathcal{K} \supseteq \mathcal{H}$ be an $\Q$-isometric dilation of $(A, X_1, X_2, X_3)$. Consider the space  given by
	\[
	\mathcal{K}_0=\overline{\text{span}}\left\{V^{i}V_1^{j}V_2^{k}V_3^{\ell}h : \ h \in \mathcal{H} \ \text{and} \ i, j, k, \ell \in \mathbb{N} \cup \{0\} \right\}.
	\] 
	Clearly, $\mathcal{K}_0$ is a joint invariant subspace for $V, V_1, V_2, V_3$ and $\HS \subseteq \mathcal{K}_0 \subseteq \mathcal{K}$. Let us define 
	$(W, W_1, W_2, W_3)=(V|_{\mathcal{K}_0}, V_1|_{\mathcal{K}_0}, V_2|_{\mathcal{K}_0}, V_3|_{\mathcal{K}_0})$. Then 
	\[
	A^{i}X_1^{j}X_2^{k}X_3^{\ell}=P_\mathcal{H}W^{i}W_1^{j}W_2^{k}W_3^{\ell}|_\mathcal{H}\]
	for all $i, j, k, \ell \in \N \cup \{0\}$. Since $(V, V_1, V_2, V_3)$ on $\mathcal{K}$ is an $\Q$-isometry, there is an $\Q$-unitary $(U, U_1, U_2, U_3)$ acting on a Hilbert space $\widetilde{\mathcal{K}}$ containing $\mathcal{K}$  such that $\mathcal{K}$ is a joint invariant subspace for $(U, U_1, U_2, U_3)$ and		$
	(V, V_1, V_2, V_3)=(U|_{\mathcal{K}}, U_1|_{\mathcal{K}}, U_2|_{\mathcal{K}}, U_3|_{\mathcal{K}})$. Thus, we have that
	\[
	(W, W_1, W_2, W_3)=(V|_{\mathcal{K}_0}, V_2|_{\mathcal{K}_0}, V_3|_{\mathcal{K}_0}, V_4|_{\mathcal{K}_0})=(U|_{\mathcal{K}_0}, U_1|_{\mathcal{K}_0}, U_2|_{\mathcal{K}_0}, U_3|_{\mathcal{K}_0})
	\]
	and so, $(W, W_1, W_2, W_3)$ on $\mathcal{K}_0$ is a minimal $\Q$-isometric dilation of $(A, X_1, X_2, X_3)$.
\end{proof}

\begin{prop}\label{prop_min_coext}
	Let $(V, V_1, V_2, V_3)$ acting on a Hilbert space $\mathcal{K}$ be an $\Q$-isometric dilation of an $\Q$-contraction $(A, X_1, X_2, X_3)$ acting on a Hilbert space $\mathcal{H}$. If $(V, V_1, V_2, V_3)$ is a minimal $\Q$-isometric dilation, then $(V^*, V_1^*, V_2^*, V_3^*)$ is an $\Q$-co-isometric extension of $(A^*, X_1^*, X_2^*, X_3^*)$.
\end{prop}

\begin{proof}
	
	Assuming the minimality of $(V, V_1, V_2, V_3)$, we have that
	\[
	\mathcal{K}=\overline{\text{span}}\left\{V^{i}V_1^{j}V_2^{k}V_3^{\ell}h : \ h \in \mathcal{H} \ \text{and} \ i, j, k, \ell \in \mathbb{N} \cup \{0\} \right\}.
	\]
	Let $h \in \mathcal{H}$. Then
	\[
	AP_\mathcal{H}(V^{i}V_1^{j}V_2^{k}V_3^{\ell}h)=A(A^{i}X_1^{j}X_2^{k}X_3^{\ell}h)=A^{i+1}X_1^{j}X_2^{k}X_3^{\ell}h=P_\mathcal{H}(V^{i+1}V_1^{j}V_2^{k}V_3^{\ell}h)=P_\mathcal{H}V(V^{i}V_1^{j}V_2^{k}V_3^{\ell}h)
	\]	
	and so, $AP_\mathcal{H}=P_\mathcal{H}V$. Also, one can show that $X_nP_\mathcal{H}=P_\mathcal{H}V_n$ for $n=1,2,3$. Moreover, for $h \in \mathcal{H}$ and $k \in \mathcal{K}$, we have that
	$
	\langle A^*h, k \rangle =\langle A^*h, P_\mathcal{H}k \rangle =\langle h, AP_\mathcal{H}k \rangle =\langle h, P_\mathcal{H}Vk \rangle =\langle V^*h, k \rangle. 
	$
	Therefore, $A^*=V^*|_\mathcal{H}$ and similarly $X_n^*=V_n^*|_\mathcal{H}$ for $n=1, 2, 3$. The proof is complete. 
\end{proof}

We now present our dilation theorem, whose proof is based on the conditional dilation result for $\E$-contractions stated in Theorem \ref{thm_pure_E_dilation}.

\begin{thm}\label{thm_Q_pure}
	Let $(A, X_1, X_2, X_3)$ be an $\Q$-contraction on a Hilbert space $\HS$ with $X_3$ being a pure contraction. Let $G_1, G_2$ be the fundamental operators of $(X_1^*, X_2^*, X_3^*)$. Suppose there exist $A_0, A_1$ in $\mathcal{B}(\mathcal{D}_{X_3^*})$ such that the following hold.
	
	\begin{minipage}[t]{0.25\textwidth}
		\smallskip
		\begin{equation*}
			\begin{split}
				& 1. \ [G_1, G_2]=0,\\
				& 4. \ [G_1, A_0]=[A_1, G_2^*] , \\
				& 7.  \ [G_2, A_0]=[A_1, G_1^*] , \\
				& 10.  \ A_0^*A_1+G_1G_2=0, \\
			\end{split}
		\end{equation*}	
	\end{minipage}
	\begin{minipage}[t]{0.35\textwidth}
		\smallskip
		\begin{equation*}
			\begin{split}
				& 2. \ [G_1^*, G_1]=[G_2^*, G_2],\\
				& 5.  \ [G_1, A_1]=0 , \\
				& 8. \ [G_2, A_1]=0,\\ 
				& 11.  \ AD_{X_3^*}=D_{X_3^*}A_0+X_3D_{X_3^*}A_1.
			\end{split}
		\end{equation*}	
	\end{minipage}
	\begin{minipage}[t]{0.3\textwidth}
		\smallskip
		\begin{equation*}
			\begin{split}
				& 3. \ [G_2^*, A_0]=0,\\
				& 6.  \ [G_1^*, A_0]=0, \\
				& 9.  \ A_0^*A_0+A_1^*A_1=I-G_1G_1^*-G_2^*G_2 , \\
			\end{split}
		\end{equation*}	
	\end{minipage}
$ $\\
Then the operator quadruple 	
	\begin{equation*}
		(V, V_1, V_2, V_3)=\left(I\otimes A_0+T_z \otimes A_1, \ I\otimes G_1^*+T_z \otimes G_2, \ I\otimes G_2^*+T_z \otimes G_1, \ T_z \otimes I  \right)
	\end{equation*}
	on $H^2(\D) \otimes \mathcal{D}_{X_3^*}$ is a minimal $\Q$-isometric dilation of $(A, X_1, X_2, X_3)$. 
\end{thm}

\begin{proof}
	The minimality follows trivially if we prove $(V, V_1, V_2, V_3)$ is an $\Q$-isometric dilation of $(A, X_1, X_2, X_3)$. This is because $V_3$ on $H^2(\D) \otimes \mathcal{D}_{X_3^*}$ is the minimal isometric dilation of $X_3$. Since $X_3$ is a pure contraction, it follows from the proof of Theorem 3.2 in \cite{Pal-IV} that the map given by
	\[
	W: \HS \to H^2(\D) \otimes \mathcal{D}_{X_3^*}, \quad Wh=\overset{\infty}{\underset{n=0}{\sum}}z^n \otimes D_{X_3^*}X_3^{*n}h.
	\]
	is an isometry. Also, for a basis vector $z^n \otimes y$ of $H^2(\D) \otimes \mathcal{D}_{X_3^*}$, we have that 
	\begin{equation}\label{eqn_W*}
		W^*(z^n \otimes y)=X_3^nD_{X_3^*}y, \quad \text{for} \ n=0, 1, 2, \dotsc.
	\end{equation} 
	Indeed, it was proved in Theorem 3.2 of \cite{Pal-IV} that $V_j^*|_{W(\HS)}=WX_j^*W^*|_{W(\HS)}$ for $j=1, 2, 3$ and so, $(V_1, V_2, V_3)$ is an $\E$-isometric dilation of $(X_1, X_2, X_3)$. Consequently, $V_2$ is a contraction and so,  $\|G_1^*+G_2z\|_{\infty, \DC} \leq 1$. It is easy to see that $(V, V_1, V_2, V_3)$ is unitarily equivalent to the quadruple $(T_{A_0+A_1z}, T_{G_1^*+G_2z}, T_{G_2^*+G_1z}, T_z)$ on $H^2(\mathcal{D}_{X_3^*})$ via the natural identification map. It is evident that conditions $(1)$-$(10)$ in the statement of this theorem and that of Theorem \ref{thm_pureQ_isom} are same. Thus, by Theorem \ref{thm_pureQ_isom}, $(V, V_1, V_2, V_3)$ is a pure $\Q$-isometry. For a basis vector $z^n \otimes y$ of $H^2(\D) \otimes \mathcal{D}_{X_3^*}$, we have that
	\begin{align*}
		\begin{split}
			W^*V(z^n \otimes y)
			=W^*(z^n \otimes A_0y)+W^*(z^{n+1}\otimes A_1y)
			&=X_3^nD_{X_3^*}A_0y+X_3^{n+1}D_{X_3^*}A_1y \qquad \quad \ [\text{by} \ \eqref{eqn_W*}]\\
			&=X_3^n(D_{X_3^*}A_0+X_3D_{X_3^*}A_1)y\\
			&=X_3^nAD_{X_3^*}y \qquad \qquad \qquad  \ \  [\text{by condition $(11)$}] \\
			&=AX_3^nD_{X_3^*}y\\
			&=AW^*(z^n \otimes y). \qquad  \qquad \qquad \qquad \quad [\text{by} \ \eqref{eqn_W*}]\\
		\end{split}
	\end{align*}
	Therefore, $W^*V=AW^*$ and so, $V^*|_{W(\HS)}=WA^*W^*|_{W(\HS)}$. The proof is now complete.
\end{proof}

We achieve the following operator model for a subclass of $\Q$-contractions as a consequence of a minimal $\Q$-isometric dilation obtained in Theorem \ref{thm_Q_pure}. In particular, any such $\Q$-contraction can be realized as the restriction of an $\Q$-co-isometry with last component being a pure co-isometry.

\begin{thm}\label{thm_func_model}
		Let $(A, X_1, X_2, X_3)$ be an $\Q$-contraction on a Hilbert space $\HS$ such that $X_3^n \to 0$ strongly as $n \to \infty$. Let $G_1, G_2$ be the fundamental operators of $(X_1, X_2, X_3)$. Suppose there exist $A_0, A_1$ in $\mathcal{B}(\mathcal{D}_{X_3})$ such that the following hold.
	
	\begin{minipage}[t]{0.25\textwidth}
		\smallskip
		\begin{equation*}
			\begin{split}
				& 1. \ [G_1, G_2]=0,\\
				& 4. \ [G_1, A_0]=[A_1, G_2^*] , \\
				& 7.  \ [G_2, A_0]=[A_1, G_1^*] , \\
				& 10.  \ A_0^*A_1+G_1G_2=0, \\
			\end{split}
		\end{equation*}	
	\end{minipage}
	\begin{minipage}[t]{0.35\textwidth}
		\smallskip
		\begin{equation*}
			\begin{split}
				& 2. \ [G_1^*, G_1]=[G_2^*, G_2],\\
				& 5.  \ [G_1, A_1]=0 , \\
				& 8. \ [G_2, A_1]=0,\\ 
				& 11.  \ A^*D_{X_3}=D_{X_3}A_0+X_3^*D_{X_3}A_1.
			\end{split}
		\end{equation*}	
	\end{minipage}
	\begin{minipage}[t]{0.3\textwidth}
		\smallskip
		\begin{equation*}
			\begin{split}
				& 3. \ [G_2^*, A_0]=0,\\
				& 6.  \ [G_1^*, A_0]=0, \\
				& 9.  \ A_0^*A_0+A_1^*A_1=I-G_1G_1^*-G_2^*G_2 , \\
			\end{split}
		\end{equation*}	
	\end{minipage}
	$ $\\
	Suppose $(N, N_1, N_2, N_3)=(V^*, V_1^*, V_2^*, V_3^*)$, where the quadruple $(V, V_1, V_2, V_3)$ is given by	
	\begin{equation*}
		(V, V_1, V_2, V_3)=\left(I\otimes A_0+T_z \otimes A_1, \ I\otimes G_1^*+T_z \otimes G_2, \ I\otimes G_2^*+T_z \otimes G_1, \ T_z \otimes I  \right)
	\end{equation*}
	on $\mathcal{K}=H^2(\D) \otimes \mathcal{D}_{X_3}$ and 	let 
	\[
	W: \HS \to H^2(\D) \otimes \mathcal{D}_{X_3}, \quad Wh=\overset{\infty}{\underset{n=0}{\sum}}z^n \otimes D_{X_3}X_3^{n}h.
	\]
	Then $(N, N_1, N_2, N_3)$ is  an $\Q$-co-isometry with $N_3$ being a pure co-isometry and $W$ is an isometry such that $(A, X_1, X_3, X_3)=(N|_{W(\HS)}, N_1|_{W(\HS)}, N_2|_{W(\HS)}, N_3|_{W(\HS)})$. 
\end{thm}

\begin{proof}
	It follows from Lemma \ref{basiclem:01} that the commuting quadruple $(B, Y_1, Y_2, Y_3)=(A^*, X_1^*, X_2^*, X_3^*)$ acting on $\HS$ is an $\Q$-contraction. By hypothesis, $Y_3$ is a pure contraction and $G_1, G_2$ are the fundamental operators of $(Y_1^*, Y_2^*, Y_3^*)$. Again by hypothesis, there exist $A_0, A_1 \in \mathcal{B}(\mathcal{D}_{Y_3^*})$ such that the operator equations associated with $(B, Y_1, Y_2, Y_3)$ in Theorem \ref{thm_Q_pure} are satisfied. We have by Theorem \ref{thm_Q_pure} that $W$ is an isometry and the operator quadruple on $H^2(\D) \otimes \mathcal{D}_{X_3}$	
	\begin{equation*}
		(V, V_1, V_2, V_3)=\left(I\otimes A_0+T_z \otimes A_1, \ I\otimes G_1^*+T_z \otimes G_2, \ I\otimes G_2^*+T_z \otimes G_1, \ T_z \otimes I  \right)
	\end{equation*}
	is a minimal $\Q$-isometric dilation of $(A^*, X_1^*, X_2^*, X_3^*)$. By Proposition \ref{prop_min_coext}, $(V^*, V_1^*, V_2^*, V_3^*)$ is an $\Q$-co-isometric extension of $(A, X_1, X_2, X_3)$, which completes the proof.
\end{proof}

Next, we provide an example of $\Q$-contraction satisfying the operator equations in Theorem \ref{thm_func_model}.

\begin{eg}
	Let $r=1\slash 2$ and consider the quadruple $(A, X_1, X_2, X_3)=\left(rI, rI, rI, 0\right)$ on a Hilbert space $\HS$. Note that $P=(p_{ij})=\begin{pmatrix}
		r & r \\
		r & r
	\end{pmatrix}$ has norm $1$ and $(r, r, r, 0)=(p_{21}, p_{11}, p_{22}, \det(P))$. By the definition of $\mathbb H_N$ as mentioned in the `Introduction', it follows that $(r, r, r, 0) \in \overline{\mathbb H}_N$. Since $\overline{\Q}_N \subseteq \CQ$ (see Theorem 6.3 in \cite{Biswas}), we have that $(r, r, r, 0) \in \CQ$ and thus $(A, X_1, X_2, X_3)$ is an $\Q$-contraction. The fundamental operators of $(X_1, X_2, X_3)$ are given by $(G_1, G_2)=(rI, rI)$. It is not difficult to see that the system of operator equations as in Theorem \ref{thm_func_model} admits a solution $(A_0, A_1)=(rI, -rI)$.  Furthermore, Theorem \ref{thm_Q_pure} ensures that a minimal pure $\Q$-isometric dilation of $(A, X_1, X_2, X_3)$ is given by 
$(V, V_1, V_2, V_3)=\left(I\otimes A_0+T_z \otimes A_1, \ I\otimes G_1^*+T_z \otimes G_2, \ I\otimes G_2^*+T_z \otimes G_1, \ T_z \otimes I  \right)$ on $H^2(\D) \otimes \mathcal{D}_{X_3}$.  Equivalently, we can write 
	\begin{small} 	
	\begin{equation*}
		V=  \begin{bmatrix} 
			rI & 0 & 0 &  \dotsc \\
			-rI & rI & 0 & \dotsc \\
			0 & -rI & rI &  \dotsc\\
			\dotsc & \dotsc & \dotsc  & \dotsc \\
		\end{bmatrix}, \ 
		V_1= 	V_2=  \begin{bmatrix} 
			rI & 0 & 0 &  \dotsc \\
			rI & rI & 0 & \dotsc \\
			0 & rI & rI &  \dotsc\\
			\dotsc & \dotsc & \dotsc  & \dotsc \\
		\end{bmatrix},  \ 
		V_3=  \begin{bmatrix} 
			0 & 0 & 0 &  \dotsc \\
			I & 0 &  0 & \dotsc \\
			0 & I & 0 &  \dotsc\\
			\dotsc & \dotsc  & \dotsc & \dotsc \\
		\end{bmatrix}
	\end{equation*}
\end{small}
with respect to the decomposition	$\ell^2(\mathcal{D}_{X_3})=\mathcal{D}_{X_3} \oplus \mathcal{D}_{X_3} \oplus \dotsc$. It is easy to see that $(A, X_1, X_2, X_3)=(V^*|_{\HS}, V_1^*|_{\HS}, V_2^*|_{\HS}, V_3^*|_{\HS})$, which is the functional model constructed in Theorem \ref{thm_func_model}.
\qed
\end{eg}

As mentioned earlier, the problem of rational dilation remains open for contractions associated with $\E$ and $\Pe$. In contrast, rational dilation holds for $\Gamma$-contractions and for commuting pairs of contractions, thereby providing canonical classes of $\Q$-contractions that admit $\Q$-isometric dilations, as proved in the next two results.

\begin{prop}
	Every $\Q$-contraction of the form $(A, 0, 0, X_3)$ admits an $\Q$-isometric dilation. \end{prop}

\begin{proof}
	Assume that $(A, 0 , 0, X_3)$ is an  $\Q$-contraction. By Theorem \ref{thm_connectII},  $(A, X_3)$ is a commuting pair of contractions.  A well-known result due to Ando (see Chapter I of \cite{Nagy}) gives that $(A, X_3)$ dilates to a pair of commuting isometries $(V, V_3)$. Finally, we have by Corollary \ref{cor_P_Q_iso}  that $(V, 0, 0, V_3)$ is an $\Q$-isometric dilation of $(A, 0, 0, X_3)$.
\end{proof}

The success of $\Gamma$-isometric dilation of a $\Gamma$-contraction provides another class of $\Q$-contractions admitting an $\Q$-isometric dilation.

\begin{prop}
	Every $\Q$-contraction of the form $(0, S\slash 2, S\slash 2, P)$ admits an $\Q$-isometric dilation. \end{prop}
\begin{proof}
	Let $(0, S\slash 2, S\slash 2, P)$ be an $\Q$-contraction on a Hilbert space $\HS$. Take a matricial polynomial $f=[f_{ij}]$ and define $g=[g_{ij}]$ with $g_{ij}(z_1, z_2)=f_{ij}(0, z_1\slash 2, z_1\slash 2, z_2)$ for $1 \leq i, j \leq n$. We have by Theorem \ref{thm_connectII} that  $(S, P)$ is a $\Gamma$-contraction. Furthermore, Theorem 4.3 in \cite{Bhattacharyya} ensures the existence of a $\Gamma$-unitary dilation $(T, U)$ of $(S, P)$. It follows from Arveson's dilation theorem (see Theorem 1.2.2 in \cite{ArvesonII}) that $\Gamma$ is a complete spectral set for $(T, U)$ and so, the matricial von Neumann's inequality \eqref{eqn_vN_matrix} holds for $(T, U)$. Also, by Theorem \ref{thm_connectII}, $(0, T\slash 2, T\slash 2, U)$ is a normal $\Q$-contraction that dilates $(0, S \slash 2, S\slash 2, P)$. Then
	\begin{align*}
		\begin{split} 
			\left\|[f_{ij}(0, S\slash 2, S\slash 2, P)]\right\|
			\leq \left\|[f_{ij}(0, T\slash 2, T\slash 2, U)]\right\|
			&=\left\|[g_{ij}(T, U)]\right\|\\
			& \leq \sup \{\|[g_{ij}(z_1, z_2)]\| : (z_1, z_2) \in \Gamma \}\\
			& =\sup \{\|[f_{ij}(0, z_1\slash 2, z_1\slash 2, z_2)]\| : (z_1, z_2) \in \Gamma \}\\
			& \leq \sup\{\|[f_{ij}(\underline{z})] \| : \underline{z} \in \CQ \},
		\end{split}
	\end{align*} 
	where the last inequality follows from Theorem \ref{thm_connect}. Hence, $\CQ$ is a complete spectral set for the $\Q$-contraction $(0, S\slash 2, S\slash 2, P)$. The desired conclusion follows from Arveson's dilation theorem.
\end{proof}

While the problem of rational dilation on $\E$ and $\Pe$ remains open, we establish an equivalent criterion for rational dilation on these domains in terms of rational dilation on the hexablock. 

\begin{prop}\label{prop_dil_E}
	An $\E$-contraction $(X_1, X_2, X_3)$ admits an $\E$-isometric dilation if and only if the $\Q$-contraction $(0, X_1, X_2, X_3)$ admits an $\Q$-isometric dilation.
\end{prop}

\begin{proof}
	Let $(X_1, X_2, X_3)$ be an $\E$-contraction and let $(N_1, N_2, N_3)$ be its $\E$-isometric dilation. Take $f=[f_{ij}]$ be a matricial polynomial and consider $g=[g_{ij}]$ with $g_{ij}(z_1, z_2, z_3)=f_{ij}(0, z_1, z_2, z_3)$. By Theorem \ref{thm_connectII}, both $(0, X_1, X_2, X_3)$ and $(0, N_1, N_2, N_3)$ are $\Q$-contractions. It follows from Arveson's dilation theorem that $\overline{\E}$ is a complete spectral set for $(X_1, X_2, X_3)$. Then
	\begin{align*}
		\|[f_{ij}(0, X_1, X_2, X_3)]\|=\|[g_{ij}(X_1, X_2, X_3)]\| 
		& \leq \sup\{\|[g_{ij}(z_1, z_2, z_3)]\|: (z_1, z_2, z_3) \in \EC\}\\
		& \leq \sup\{\|[f_{ij}(0, z_1, z_2, z_3)]\|: (z_1, z_2, z_3) \in \EC \}\\
		& \leq \sup\{\|[f_{ij}(\underline{z})]\| : \underline{z} \in \CQ\},
		\end{align*} 
where we have used Theorem \ref{thm_connect} in the last inequality. Consequently, $\CQ$ is a complete spectral set for $(0, X_1, X_2, X_3)$ and by Arveson's dilation theorem, $(0, X_1, X_2, X_3)$ admits an $\Q$-isometric dilation. For the converse, one can directly apply Theorems \ref{thm_connectII} \& \ref{thm_Q_isometry}. 
\end{proof}

We have an analogous result for $\B_2$-contractions. 

\begin{prop}
	A $\B_2$-contraction $(A, X_1)$ admits a $\B_2$-isometric dilation if and only if $(A, X_1, 0, 0)$ admits an $\Q$-isometric dilation. 
\end{prop}

\begin{proof}
	Let $(A, X_1)$ be a $\B_2$-contraction and let $(V, V_1)$ be its $\B_2$-isometric dilation. Take $f=[f_{ij}]$ be a matricial polynomial and consider $g=[g_{ij}]$ with $g_{ij}(z_1, z_2)=f_{ij}(z_1, z_2, 0, 0)$. By Theorem \ref{thm_connectII}, $(A, X_1, 0, 0)$ and $(V, V_1, 0, 0)$ are $\Q$-contractions. Since $(A, X_1)$ admits a rational dilation to $(V, V_1)$, we have by Arveson's dilation theorem that  $\BC$ is a complete spectral set for $(A, X_1)$. Then
	\begin{align*}
		\begin{split} 
			\left\|[f_{ij}(A, X_1, 0, 0)]\right\|
			=\left\|[g_{ij}(A, X_1)]\right\|
			& \leq \sup \{\|[g_{ij}(z_1, z_2)]\| : (z_1, z_2) \in \BC \}\\
			& =\sup \{\|[f_{ij}(z_1, z_2,0, 0)]\| : (z_1, z_2) \in \BC \}\\
			& \leq \sup\{\|[f_{ij}(\underline{z})] \| : \underline{z} \in \CQ \},
		\end{split}
	\end{align*} 
	where the last inequality follows from Theorem \ref{thm_connect}. Hence, $\CQ$ is a complete spectral set for $(A, X_1, 0, 0)$ and so, by Arveson's dilation theorem, it admits an $\Q$-isometric dilation. The converse follows directly from Theorems \ref{thm_connectII} \& \ref{thm_Q_isometry}. 
\end{proof}

Similarly, one can prove the next result as an application of Theorem \ref{thm_connectII} and Corollary \ref{cor_P_Q_iso}.

\begin{prop}\label{prop_dil_P}
	A $\Pe$-contraction $(A, S, P)$ admits a $\Pe$-isometric dilation if and only if the $\Q$-contraction $(A, S\slash 2, S\slash 2, P)$ admits an $\Q$-isometric dilation.
\end{prop}
At present, no domain in $\C^3$ is known for which the rational dilation problem has an affirmative answer. Our wild guess is that rational dilation fails on the tetrablock and the pentablock. However, Propositions \ref{prop_dil_E} \& \ref{prop_dil_P} show that if rational dilation fails on either the tetrablock or the pentablock, then it must also fail on the hexablock.

	\section{Canonical decomposition of an $\Q$-contraction}\label{sec_06}
	
	\vspace{0.2cm}

	\noindent A canonical decomposition of contraction (see Theorem 3.2 in Chapter I of \cite{Nagy}) that every contraction $T$ on a Hilbert space $\HS$ admits a canonical decomposition $T_1\oplus T_2$ with respect to $\HS=\HS_1 \oplus \HS_2$, where $T_1$ is a unitary and $T_2$ is a completely non-unitary contraction. The maximal reducing subspace $\HS_1$ on which $T$ acts as a unitary is given by
	\begin{equation*}
		\begin{split}
			\mathcal{H}_1 =\{h \in \mathcal{H}: \|T^nh\|=\|h\|=\|T^{*n}h\|, \ n=1,2, \dotsc \} = \underset{n \in \mathbb{Z}}{\bigcap} Ker D_{T(n)},\\	
		\end{split}
	\end{equation*}
	where
	\[
	D_{T(n)}= \left\{
	\begin{array}{ll}
		(I-T^{*n}T^n)^{1\slash 2} & n \geq 0 \\
		(I-T^{|n|}T^{*|n|})^{1\slash 2} & n <0 \,.\\
	\end{array} 
	\right. 
	\]
	A similar result is true for a doubly commuting pair of contractions as the following result shows.
	
	\begin{thm}[\cite{Pal-II}, Theorem 4.2]\label{lem5.2}
		For a pair of doubly commuting contractions $P,Q$ acting on a Hilbert space $\HS$, if $Q=Q_1 \oplus Q_2$ is the canonical decomposition of $Q$ with respect to the orthogonal decomposition $\HS=\HS_1 \oplus \HS_2$, then $\mathcal{H}_1, \mathcal{H}_2$ are reducing subspaces for $P$.
	\end{thm}
	
	We recall from \cite{Pal2016} an analogue of canonical decomposition for an $\E$-contraction $(X_1, X_2, X_3)$. In fact, such a decomposition of $(X_1, X_2, X_3)$ is can be obtained through the canonical decomposition of the contraction $X_3$ as the following theorem shows.
	
	\begin{thm}[\cite{Pal2016}, Theorem 3.1]\label{thm5.3}
		Let $(X_1, X_2, X_3)$ be an $\E$-contraction on a Hilbert space $\mathcal{H}$. Let $\mathcal{H}_1$ be the maximal subspace of $\mathcal{H}$ which reduces $X_3$ and on which $X_3$ is unitary. Let $\mathcal{H}_2=\mathcal{H} \ominus \mathcal{H}_1$. Then $\mathcal{H}_1, \mathcal{H}_2$ reduce $X_1, X_2$. Moreover, $(X_1|_{\mathcal{H}_1}, X_2|_{\mathcal{H}_1}, X_3|_{\mathcal{H}_1})$ is an $\E$-unitary and $(X_1|_{\mathcal{H}_2}, X_2|_{\mathcal{H}_2}, X_3|_{\mathcal{H}_2} )$ is a completely non-unitary $\E$-contraction.
	\end{thm}
	
	We now prove a canonical decomposition of an $\Q$-contraction. The proof is divided into two parts where in the first part we obtain the decomposition result for a normal $\Q$-contraction.	
	
	\begin{prop}\label{normal}
		Let $(N_0, N_1, N_2, N_3)$ be a normal $\Q$-contraction acting on a Hilbert space $\HS$. Then there exists an orthogonal decomposition $\mathcal{H}=\mathcal{H}^{(u)}\oplus \mathcal{H}^{(c)}$ into joint reducing subspaces $\mathcal{H}^{(u)}$ and $\mathcal{H}^{(c)}$ of $N_0, N_1, N_2, N_3$ such that the following hold.
		\begin{enumerate}
			\item $(N_0|_{\HS^{(u)}}, N_1|_{\HS^{(u)}}, N_2|_{\HS^{(u)}}, N_3|_{\HS^{(u)}})$ is an $\Q$-unitary.
			\item  $(N_0|_{\HS^{(c)}}, N_1|_{\HS^{(c)}}, N_2|_{\HS^{(c)}}, N_3|_{\HS^{(c)}})$ is a completely non-unitary $\Q$-contraction.
		\end{enumerate}
		Moreover, $\mathcal{H}^{(u)}$ is the maximal closed joint reducing subspace of $N_0, N_1, N_2, N_3$ restricted to which $(N_0, N_1, N_2, N_3)$ is an $\Q$-unitary.
	\end{prop}
	
	\begin{proof}
		We have by Proposition \ref{prop2.3} that $(N_1, N_2, N_3)$ is an $\E$-contraction acting on $\HS$. Assume that $\mathcal{H}=\mathcal{H}_1\oplus \mathcal{H}_2$ is the canonical decomposition of the contraction $N_3$. A simple application of Theorem \ref{lem5.2} gives that $\mathcal{H}_1, \mathcal{H}_2$ are joint reducing subspaces for $N_0, N_1, N_2, N_3$. Suppose that 
		\[
		N_0=\begin{bmatrix}
			P_0 & 0 \\
			0 & Q_0
		\end{bmatrix}, \quad N_1=\begin{bmatrix}
			P_1 & 0 \\
			0 & Q_1
		\end{bmatrix}, \quad N_2=\begin{bmatrix}
			P_2 & 0 \\
			0 & Q_2
		\end{bmatrix} \quad \text{and} \quad N_3=\begin{bmatrix}
			P_3 & 0 \\
			0 & Q_3
		\end{bmatrix}
		\]
		with respect to $\mathcal{H}=\mathcal{H}_1\oplus \mathcal{H}_2$. Note that $P_3$ is a unitary and $Q_3$ is a completely non-unitary contraction. We have by Theorem \ref{thm5.3} that $(P_1, P_2, P_3)$ is an $\E$-unitary on $\mathcal{H}_1$. Let us define 
		\[
		\mathcal{H}^{(u)}= Ker(I-P_0^{*}P_0-P_1^*P_1)=\{x \in \HS_1 : P_0^*P_0x+P_1^*P_1x=x  \}.
		\] 
		We have by Fuglede's theorem that $(P_0, P_1, P_2, P_3)$ is a doubly commuting quadruple of operators. Consequently, $\HS^{(u)}$ is a reducing subspace for $P_j$ and so, for $N_j$ with $j=0, 1, 2, 3$. Let us define $\underline{U}=(U_0, U_1, U_2, U_3 )= (P_0|_{\HS^{(u)}}, P_1|_{\HS^{(u)}}, P_2|_{\HS^{(u)}}, P_3|_{\HS^{(u)}})$. Then $(P_1|_{\HS^{(u)}}, P_2|_{\HS^{(u)}}, P_3|_{\HS^{(u)}})$ is an $\E$-unitary and $U_0^*U_0x+U_1^*U_1x=P_0^*P_0x+P_1^*P_1x=x$ for all $x \in \HS^{(u)}$. By Theorem \ref{thm_Q_unitary}, $\underline{U}$ is an $\Q$-unitary. Let $\HS' \subseteq \HS$ be a joint reducing subspace for $N_0, N_1, N_2, N_3$ such that $N'=(N_0|_{\HS'}, N_1|_{\HS'}, N_2|_{\HS'}, N_3|_{\HS'})$ is an $\Q$-unitary. Let $N_j'=N_j|_{\HS'}$ for $0 \leq j \leq 3$. By Theorem \ref{thm_Q_unitary}, $(N_1, N_2, N_3)$ is an $\E$-unitary and so, $N_3$ is a unitary. Since $\HS_1$ is the maximal closed subspace of $\HS$ that reduces $N_3$ to unitary, we have that $\HS' \subseteq \HS_1$. Consequently, $N_j|_{\HS'}=P_j|_{\HS'}$ for $j=0, 1$. Since $(N_0, N_1, N_2, N_3)$ on $\mathcal{H}'$ is an $\Q$-unitary, we have by Theorem \ref{thm_Q_unitary} that $P_0^*P_0x+P_1^*P_1x=N_0^*N_0x+N_1^*N_1x=x$ for all $x \in \HS'$. Hence, $\HS' \subseteq \HS^{(u)}$ and so, $\mathcal{H}^{(u)}$ is the maximal closed joint reducing subspace for $N_0, N_1, N_2, N_3$ restricted to which $(N_0, N_1, N_2, N_3)$ is an $\Q$-unitary. Let $\mathcal{H}^{(c)}=\HS \ominus \HS^{(u)}$. The desired conclusion now follows from the maximality of $\HS^{(u)}$.
	\end{proof}
	
	We now present the main theorem of this section. For the sake of brevity, we denote by 
	\[
	\underline{T}^s=T_1^{s_1} \dotsc T_n^{s_n} \qquad \text{and}  \qquad \underline{T}^{*s}=(T_1^*)^{s_1} \dotsc (T_n^*)^{s_n}
	\]
	for  $s=(s_1, \dotsc, s_n) \in (\mathbb{N} \cup \{0\})^n$ and a commuting operator tuple $\underline{T}=(T_1, \dotsc, T_n)$. 
	\begin{thm}\textbf{(Canonical decomposition of an $\Q$-contraction).}\label{thm_can_Q}
		Let $(A, X_1, X_2, X_3)$ be an $\Q$-contraction acting on a Hilbert space $\HS$. Then there exists an orthogonal decomposition $\mathcal{H}=\mathcal{H}^{(u)}\oplus \mathcal{H}^{(c)}$ into joint reducing subspaces $\mathcal{H}^{(u)}, \mathcal{H}^{(c)}$ for $A, X_1, X_2, X_3$ such that 
		\begin{enumerate}
			\item $(A|_{\HS^{(u)}}, X_1|_{\HS^{(u)}}, X_2|_{\HS^{(u)}}, X_3|_{\HS^{(u)}})$ is an $\Q$-unitary and 
			\item $(A|_{\HS^{(c)}}, X_1|_{\HS^{(c)}}, X_2|_{\HS^{(c)}}, X_3|_{\HS^{(c)}})$ is a completely non-unitary $\Q$-contraction.
		\end{enumerate}
		 Moreover, $\mathcal{H}^{(u)}$ is the maximal closed joint reducing subspace for $A, X_1, X_2, X_3$ restricted to which $(A, X_1, X_2, X_3)$ is an $\Q$-unitary.
	\end{thm}
	
	\begin{proof}
		Let $\underline{T}=(A, X_1, X_2, X_3)$ be an $\Q$-contraction on a Hilbert space $\HS$. Define
		\[
		\mathcal{H}_0=\bigcap_{s \in \mathbb{N}^4}\bigcap_{t \in \mathbb{N}^4}Ker\left(\underline{T}^s\underline{T}^{*t}-\underline{T}^{*t}\underline{T}^{s}\right).
		\]
		We have by Corollary 4.2 in \cite{Eschmeier} that $\mathcal{H}_0$ is the largest joint reducing subspace of $A, X_1, X_2, X_3$ restricted to which $(A, X_1, X_2, X_3)$ is a commuting quadruple of normal operators. Let us define $\underline{N}=(N_0, N_1, N_2, N_3)=(A|_{\HS_0}, X_1|_{\HS_0}, X_2|_{\HS_0}, X_3A|_{\HS_0})$ on $\HS_0$.	Then $\underline{N}$ is a normal $\Q$-contraction. By Proposition \ref{normal}, there is a maximal closed joint reducing subspace $\HS^{(u)}$ for $N_0, N_1, N_2, N_3$ contained in $\HS_0$ such that $(N_0|_{\mathcal{H}^{(u)}}, N_1|_{\mathcal{H}^{(u)}}, N_2|_{\mathcal{H}^{(u)}}, N_3|_{\HS^{(u)}})$ is an $\mathbb{H}$-unitary. One can employ similar method as in Theorem \ref{normal} and prove that $\HS^{(u)}$ is the maximal closed joint reducing subspace of $A, X_1, X_3, X_4$ restricted to which $\underline{T}$ is an $\Q$-unitary. Let $\HS^{(c)}=\HS \ominus \HS^{(u)}$. The remaining part of the theorem follows from the maximality of $\HS^{(u)}$. 
	\end{proof}
	
	\section{Data Availability Statement}
	
	\noindent (1) Data sharing is not applicable to this article, because, as per our
	knowledge no datasets were generated or analysed during the current study.
	
	\smallskip
	
	\noindent (2) In case any datasets are generated and/or analysed during the current
	study which go unnoticed, they must be available from the corresponding author on
	reasonable request.

	\section{Declarations}
	
	\noindent \textbf{Ethical Approval.} This declaration is not applicable.
	
	\smallskip
	
	\noindent \textbf{Competing interests.} There are no competing interests.
	
	\smallskip
	
	\noindent \textbf{Authors' contributions.} All authors have contributed equally.
	
	\vspace{0.2cm}
	
	\noindent \textbf{Funding.} The first named author is supported by ``Core Research
	Grant'' of Science and Engineering Research Board (SERB), Govt. of India, with
	Grant No. CRG/2023/005223 and the ``Early Research Achiever Award Grant'' of IIT
	Bombay with Grant No. RI/0220-10001427-001. The second named author was supported by the Institute Postdoctoral Fellowship (IPDF) of IIT Bombay during the course of the paper.

\end{document}